\documentclass[12pt]{article}

\usepackage{geometry}
\usepackage{amsthm}
\usepackage{amsmath}
\usepackage{amssymb}
\usepackage{mathtools}
\usepackage{mathrsfs}
\usepackage{enumitem}
\usepackage{subfigure}
\usepackage{comment}
\usepackage{multicol}
\usepackage{graphicx}
\usepackage{tikz}
\usetikzlibrary[calc,intersections]
\usetikzlibrary{cd}
\usepackage{hyperref}

\newcommand{\bb}{\mathbb}
\newcommand{\mc}{\mathcal}
\newcommand{\ms}{\mathscr}
\newcommand{\mf}{\mathfrak}
\newcommand{\abs}[1]{\lvert #1 \rvert}

\newcommand{\into}{\hookrightarrow}

\DeclareMathOperator*{\conv}{conv}
\DeclareMathOperator{\ind}{index}
\DeclareMathOperator{\Subd}{Subd}
\DeclareMathOperator{\Cay}{Cay}
\DeclareMathOperator{\Span}{span}
\DeclareMathOperator{\id}{id}

\DeclareMathOperator{\plus}{Sum}
\DeclareMathOperator{\Sum}{Sum}
\DeclareMathOperator{\Forget}{Forget}

\DeclareMathOperator{\triv}{triv}
\DeclareMathOperator{\stell}{stell}

\DeclareMathOperator{\Max}{Max}
\DeclareMathOperator{\supp}{supp}
\DeclareMathOperator{\pull}{pull}
\DeclareMathOperator{\Pull}{Pull}
\DeclareMathOperator{\dice}{dice}
\DeclareMathOperator{\Dice}{Dice}

\theoremstyle{definition}

\newtheorem{thm}{Theorem}[section]
\newtheorem{prop}[thm]{Proposition}

\newtheorem{exam}[thm]{Example}

\numberwithin{equation}{section}

\newlist{results}{enumerate}{1}
\setlist[results]{label=(\alph*)}

\title{Unimodular triangulations of sufficiently large dilations}
\author{Gaku Liu}

\begin{document}
\maketitle

\begin{abstract}
An integral polytope is a polytope whose vertices have integer coordinates. A unimodular triangulation of an integral polytope in $\bb R^d$ is a triangulation in which all simplices are integral with volume $1/d!$. A classic result of Knudsen, Mumford, and Waterman states that for every integral polytope $P$, there exists a positive integer $c$ such that $cP$ has a unimodular triangulation. We strengthen this result by showing that for every integral polytope $P$, there exists $c$ such that for every positive integer $c' \ge c$, $c'P$ admits a unimodular triangulation. This answers a longstanding question in the area.
\end{abstract}

\section{Introduction}

Unimodular triangulations are elementary objects which arise naturally in algebra and combinatorics. We refer to the paper by Haase et al.\ \cite{HPPS} for an extensive survey on the subject. In this paper we answer a longstanding question on the existence of certain unimodular triangulations.

An integral polytope is a polytope with integer coordinates. Let $P$ be a $d$-dimensional integral polytope in $\bb R^d$. A unimodular triangulation of $P$ is a triangulation of $P$ into simplices each of which has the minimum possible volume $1/d!$. For $d \ge 3$, not every integral polytope has a unimodular triangulation. For example, the simplex with vertices $(0,0,0)$, $(1,0,0)$, $(0,1,0)$, $(1,1,99)$ does not have any nontrivial triangulation, but has volume $>1/6$. On the other hand, every polytope has a \emph{dilation} which admits a unimodular triangulation, as described below.

\begin{thm}[\cite{KKMS}, 1973] \label{thm:KMW}
For every integral polytope $P$, there is a positive integer $c$ such that $cP$ admits a unimodular triangulation.
\end{thm}

This theorem is known as the Kempf--Knudsen--Mumford--Saint-Donat theorem (KKMS) or the Knudsen--Mumford--Waterman theorem (KMW).\footnote{The result appears in the book \cite{KKMS} which is authored by Kempf, Knudsen, Mumford, and Saint-Donat. The individual chapter in which it appears is authored by Knudsen, who also credits Mumford and Waterman. As a result, both naming conventions have appeared in the literature. In this paper we use the convention ``KMW theorem''.} It is one of the earliest considerations of unimodular triangulations, and was proved in the context of algebraic geometry in order to prove semistable reduction for families of varieties over a curve. A more general version of semistable reduction was conjectured by Abramovich and Karu in 2000 \cite{AK00} and was reduced to proving the existence of certain unimodular triangulations of maps. The conjecture was recently proven by Adiprasito, Temkin, and the author in \cite{ALT18}.

Understanding what values of $c$ work in Theorem~\ref{thm:KMW} is an old and difficult problem. The answer is almost completely known in dimensions $\le 3$ \cite{SZ13}, and a general upper bound for the smallest possible $c$ is known in terms of the dimension and volume of the polytope \cite{HPPS}. In this paper, we prove the following:

\begin{thm} \label{thm:main}
For every integral polytope $P$, there is a positive integer $c$ such that for all $c' \ge c$, $c'P$ admits a unimodular triangulation.
\end{thm}

The result may be a bit surprising, as there are known to be polytopes $P$ and integers $c$ such that $cP$ has a unimodular triangulation but $(c+1)P$ does not.
We do not provide an explicit value for the $c$ in Theorem~\ref{thm:main}, but we note that the upper bound should be doubly exponential in the dimension and volume of $P$ using the arguments from \cite{HPPS}. We also note that the unimodular triangulation can be made regular, but to keep the paper simpler we do not prove this.

The idea of the proof is as follows. The argument from \cite{KKMS} in fact shows that if $cP$ has a unimodular triangulation, then $c'P$ has a unimodular triangulation for any $c'$ a multiple of $c$. For our result, we prove the following: There exist relatively prime positive integers $a$ and $b$ such that for any nonnegative integers $r$ and $s$, $(ra + sb)P$ has a unimodular triangulation.
In order to prove this, we extend the results of \cite{ALT18} to \emph{mixed subdivisions}, which can be thought of as coupled subdivisions of two polytopes. This ends up being more difficult than one might expect. While the constructions from \cite{ALT18} have natural mixed analogues, these natural analogues do not lead to a proof of the theorem; see the discussion at the beginning of Section~\ref{sec:main}. Therefore, we have to create some more complicated analogues as well as make some new constructions.

Another feature of this proof is that it makes heavy use of an idea we call \emph{canonical subdivisions}. Canonical subdivisions are present in a more specialized form in \cite{KKMS} and \cite{HPPS}, and are expanded in a very general way in \cite{ALT18}. In \cite{ALT18}, the idea was explained using the language of categories and functors. In this paper we rework these ideas in terms of posets and poset maps, which turns out to be more flexible and general for this situation.
Essentially, a canonical subdivision is a rule to subdivide every polytope within a family of polytopes so that this rule is compatible with the operation of restricting to a face of a polytope. The importance of this is that canonical subdivisions can be used to further subdivide arbitrary polytopal complexes in a consistent way. This allows us to iteratively construct a unimodular triangulation through many intermediate canonical subdivisions. This idea is implicit in Knudsen et al.'s original proof of the KMW theorem, and in fact they proved that any polytopal \emph{complex} $X$ of integral polytopes has a constant $c$ such that $cX$ has a unimodular triangulation. Our main result also extends to polytopal complexes in this way, which is immediate from the proof.

The cornerstone of our canonical subdivision method is Theorem~\ref{thm:confluence}, which may be of independent interest. This theorem gives conditions under which a family of non-canonical subdivisions can be used to recursively construct a canonical subdivision. This theorem is important because our desired canonical subdivision is very complicated and canonicity is difficult to check. The theorem allows us to instead construct simpler non-canonical subdivisions, after which the criteria of the theorem are straightforward to check. To demonstrate Theorem~\ref{thm:confluence}, we have also included in Section~\ref{sec:canonicalexamples} some examples of well-known subdivisions in the literature that can be constructed using these methods.

Finally, we would like to mention a few open problems. Despite the method of proof used in this paper, it is unknown whether $c_1P$ and $c_2P$ having unimodular triangulations implies that $(c_1+c_2)P$ has one as well. In addition, it is unknown whether for every dimension $d$ there exists an integer $c_d$ such that $c_d P$ has a unimodular triangulation for every $d$-dimensional polytope $P$. Finally, the question of whether specific classes of polytopes have unimodular triangulations has attracted significant attention. Classes of interest include smooth polytopes, matroid polytopes, and parallelepipeds.

The paper is organized as follows. In Section~\ref{sec:prelim} we introduce the language of polytopes, posets, and canonical subdivisions. This section is very abstract, but the author believes the initial investment makes the main argument much easier to follow. In Section~\ref{sec:Cayleypolytopes}, we introduce Cayley polytopes which are the main building blocks of our constructions. In Section~\ref{sec:boxpoints} we introduce box points, which allow us to modify triangulations to produce triangulations with simplices of smaller volume. Our main constructions and proof are in Section~\ref{sec:main}.

%\subsection*{Acknowledgments}
%
%The author would like to thank Francisco Santos for his comments on the paper. Part of this material is based upon work supported by the National Science Foundation under Award No.\ 1803638.

\section{Preliminaries} \label{sec:prelim}

\subsection{Polytopes} \label{sec:polytopes}

Throughout the paper, we work in $\bb R^d$ with some fixed $d$ unless otherwise specified. In this paper, a \emph{polytope} is a nonempty convex hull of finitely many points in $\bb R^d$. Given a polytope $P$ and a linear functional $\phi \in (\bb R^d)^\ast$, the \emph{face} of $P$ with respect to $\phi$ is the set of all points in $P$ at which $\phi$ reaches its maximum on $P$. We \textbf{do not} consider the empty set to be a face. A \emph{simplex} is the convex hull of a nonempty, affinely independent set of points.

For any polytope $P \subset \bb R^d$, we define $V(P)$ to be the real span of the set $\{ u -v : u,v \text{ are vertices of } P \}$. We say that polytopes $P_1$, \dots, $P_n$ are \emph{independent} if $V(P_1)$, \dots, $V(P_n)$ are linearly independent vector subspaces, i.e. $\dim(V(P_1) + \dots + V(P_n)) = \dim V(P_1) + \dots + \dim V(P_n)$. A polytope of the form $\sum_{j=1}^n S_j$ where $S_1$, \dots, $S_n$ are independent simplices is called a \emph{polysimplex} or \emph{product of simplices}.

%For any polytope $P$, we define $V(P)$ to be the vector subspace of $\bb R^d$ generated by the set $\{ u -v : u,v \text{ are vertices of } P \}$. We let $\dim(P) := \dim V(P)$. %We say polytopes $P_1$, \dots, $P_n$ are \emph{affinely independent} if the subspaces $V(P_1)$, \dots, $V(P_n)$ are linearly independent.

In this paper, a \emph{lattice} is an additive subgroup of $\bb Z^d$. We define the \emph{normalized index}, or just \emph{index}, of a lattice $L$ to be the group index $[ \Span_{\bb R}(L) \cap \bb Z^d, L ]$. This index is always finite. We denote the index by $\ind(L)$.

An \emph{integral polytope} is a polytope whose vertices have integer coordinates. Given an integral polytope $P$, we define $L(P)$ to be the lattice generated over $\bb Z$ by the set $\{ u -v : u,v \text{ are vertices of } P \}$. We define $N(P)$ to be the lattice $V(P) \cap \bb Z^d$. The \emph{index} of $P$ is the normalized index of $L(P)$, which equals $[N(P), L(P)]$. A \emph{unimodular simplex} is an integral simplex of index 1. For $d$-dimensional simplices in $\bb R^d$ this is equivalent to the definition given in the introduction; the current definition extends this to simplices with dimension lower than the ambient space.

An \emph{ordered polytope} is a polytope along with a total ordering on its vertices. An \emph{ordered face}, or \emph{face}, of an ordered polytope is a face of the underlying polytope along with the vertex order induced by the original polytope. Any translation or positive dilation of an ordered polytope is also an ordered polytope, with the obvious ordering.

%Let $P_1$, \dots, $P_n$ be affinely independent polytopes. Let $P := \sum_{j=1}^n P_j$. If each $P_j$ is an ordered polytope, then $P$ is an ordered polytope as follows: There is a bijection between the set of $n$-tuples $(x_1,\dots,x_n)$, where each $x_j$ is a vertex of $P_j$, to the set of vertices of $P$, given by $(x_1,\dots,x_n) \mapsto \sum_{j=1}^n x_j$. We order the tuples $(x_1,\dots,x_n)$ lexicographically, i.e., $(x_1,\dots,x_n) < (y_1,\dots,y_n)$ if there is some $j$ such that $x_i = y_i$ for all $i < j$ and $x_j < y_j$.

%Given two polytopes $P$ and $Q$, we write $P \equiv Q$ to say that $P$ is a translation of $Q$. We write $P \sqsubset Q$ to say that some translation of $P$ is contained in $Q$. %If $P$ and $Q$ are ordered polytopes, we write $P \equiv Q$ to mean that there is a translation $P \to Q$ that preserves ordering.

%Let $P$ be a non-unimodular simplex. A \emph{box point} of $P$ is any nonzero element of $N(P) / L(P)$.  If $F$ is a face of $P$, then there is a natural inclusion $N(F) / L(F) \into N(P) / L(P)$, and we identify box points of $F$ with the box points of $P$ they map to.

%For each box point $\bf{m}$ of $P$, let $c(P,\bf{m})$ be the smallest integer $c$ such that $cP$ contains a representative of $\bf{m}$.

\subsection{Posets and subdivisions}

\subsubsection{Relative posets}

Recall that a \emph{poset} is a set $\mc A$ along with a binary relation $\le_{\mc A}$ on $\mc A$ which is reflexive, antisymmetric, and transitive. We will always denote a poset by its set of elements, and if there is no risk of confusion we will use the symbol ``$\le$'' in place of $\le_{\mc A}$.

Given two posets $\mc A$ and $\mc B$, a \emph{poset map} is a function $f : \mc A \to \mc B$ such that $f(x) \le_{\mc B} f(y)$ whenever $x \le_{\mc A} y$. A \emph{poset isomorphism} is a poset map which has an inverse which is a poset map.

Let $\mc B$ be a poset. We define a \emph{$\mc B$-poset} to be a pair $(\mc A, p)$ where $\mc A$ is a poset and $p : \mc A \to \mc B$ is a poset map. If there is no risk of confusion, we denote $(\mc A, p)$ by just $\mc A$. Clearly, if $(\mc A, p)$ is a $\mc B$-poset and $(\mc A', p')$ is an $\mc A$-poset, then $(\mc A', p \circ p')$ is a $\mc B$-poset.

Given two $\mc B$-posets $(\mc A, p)$ and $(\mc A', p')$, we define a \emph{poset map over $\mc B$} to be a poset map $f : \mc A \to \mc A'$ such that $p = p' \circ f$.

Given a poset $\mc A$ and a subset $X \subset \mc A$, we let $\langle X \rangle_{\mc A}$ or $\langle X \rangle$ denote the set $X$ along with all members of $\mc A$ below $X$ in $\mc A$. If $x$ is a single element of $\mc A$, we use $\langle x \rangle$ as shorthand for $\langle \{x\} \rangle$. We let $\Max_{\mc A} X$ or $\Max X$ denote the maximal elements of $X$ with respect to $\le_{\mc A}$.

\subsubsection{The poset of polytopes and subdivisions}

Let $\mc P$ be the poset whose elements are all polytopes in $\bb R^d$, and with the partial order $F \le_{\mc P} P$ if $F$ is a face of $P$. The main class of posets we will work with in this paper are $\mc P$-posets.

\begin{exam}
Trivially, $(\mc P, \id)$ is a $\mc P$-poset, where $\id$ is the identity map.
\end{exam}

\begin{exam}
Let $\mc O$ be the poset whose elements are ordered polytopes and with the partial order $F \le_{\mc O} P$ if $F$ is an ordered face of $P$. Then $(\mc O, \Forget)$, where $\Forget$ is the map $\mc O \to \mc P$ which forgets the vertex ordering, is a $\mc P$-poset.
\end{exam}

\begin{exam} \label{exam:pointsets}
Let $\mc S$ be the poset whose elements are totally ordered finite subsets of $\bb R^d$, with the partial order $B \le_{\mc S} A$ if $B = A \cap F$, where $F$ is any face of $\conv(A)$, and the order on $B$ is the order induced by $A$. Then $(\mc S, \conv)$ is a $\mc P$-poset.
\end{exam}

A \emph{polytopal complex}, or \emph{$\mc P$-complex}, is a finite subset $X$ of $\mc P$ satisfying the following.
\begin{enumerate}[label=(\alph*)]
\item If $F$, $P \in \mc P$ such that $P \in X$ and $F \le P$, then $F \in X$.
\item If $P$, $Q \in X$ are different, then the relative interiors of $P$ and $Q$ are disjoint.
\end{enumerate}
We define the \emph{support} of a polytopal complex $X$ to be
\[
\abs{X} := \bigcup_{x \in X} x.
\]
If $\abs{X}$ is a polytope $Q$, we say that $X$ is a \emph{subdivision} of $Q$. A \emph{triangulation} is a subdivision all of whose elements are simplices.

More generally, let $(\mc A, p)$ be a $\mc P$-poset. An \emph{$(\mc A, p)$-complex} or \emph{$\mc A$-complex} is subset $X$ of $\mc A$ such that $p(X)$ is a polytopal complex and the map $p : X \to p(X)$ is a poset isomorphism from the subposet of $\mc A$ induced on $X$ to the subposet of $\mc P$ induced on $p(X)$. A \emph{subcomplex} of $X$ is a subset of $X$ which is also an $\mc A$-complex. The \emph{support} of $X$ is defined to $\abs{p(X)}$. We denote the support by simply $\abs{X}$. If the support is a polytope $Q$, we say that $X$ is a \emph{$(\mc A,p)$-subdivision} or \emph{$\mc A$-subdivision} of $Q$.

For any $\mc P$-poset $(\mc A,p)$, let $\Subd(\mc A)$ be the poset whose elements are $\mc A$-subdivisions of polytopes, and with partial order $X' \le_{\Subd(\mc A)} X$ if $X'$ is a subcomplex of $X$ such that $\abs{X'}$ is a face of $\abs{X}$. Then the map $\abs{\cdot} : \Subd(\mc A) \to \mc P$ which sends $X$ to $\abs{X}$ is a poset map, and $(\Subd(\mc A), \abs{\cdot})$ is a $\mc P$-poset.

The following proposition is easy to verify.

\begin{prop} \label{prop:mapofsubdivisions}
Let $(\mc A, p)$ and $(\mc A', p')$ be $\mc P$-posets and let $f : \mc A \to \mc A'$ be a poset map over $\mc P$. Then the map $f : \Subd(\mc A) \to \Subd(\mc A')$ which sends $X$ to $f(X)$ is a poset map over $\mc P$.
\end{prop}

\subsubsection{Trivial subdivisions and perfect posets}

For any polytope $P$, we define the \emph{trivial subdivision} of $P$ to be $\triv(P) := \langle P \rangle_{\mc P}$.

Let $(\mc A, p)$ be a $\mc P$-poset. If, for all $x \in \mc A$, the set $\langle x \rangle_{\mc A}$ is an $\mc A$-complex, then we call $(\mc A,p)$ a \emph{perfect} $\mc P$-poset. In this situation we define $\triv_{\mc A}(x) := \langle x \rangle_{\mc A}$. We necessarily have that $p(\triv_{\mc A}(x))$ is the trivial subdivision of $p(x)$.

Every example we have given so far is a perfect $\mc P$-poset. In particular, $(\Subd(\mc A), \abs{\cdot})$ is a perfect $\mc P$-poset for any $\mc P$-poset $\mc A$.

\subsection{Mixed subdivisions} \label{sec:mixed}

Let $n$ be a positive integer. Define $n\mc P$ to be the poset whose set of elements is
\[
\underbrace{\mc P \times \dots \times \mc P}_{n \text{ times}}
\]
and with $(F_1, \dots, F_n) \le_{n \mc P} (P_1, \dots, P_n)$ if and only if $F_1$, \dots, $F_n$ are faces of $P_1$, \dots, $P_n$, respectively, such that $F_1+\dots+F_n$ is a face of $P_1+\dots+P_n$. Equivalently, $(F_1, \dots, F_n) \le_{n \mc P} (P_1, \dots, P_n)$ if and only if there exists a linear functional such that $F_1$, \dots, $F_n$ are the faces of $P_1$, \dots, $P_n$, respectively, with respect to this linear functional. In the future, whenever we write $F \le P$ for two $n$-tuples of polytopes $F$, $P$, we mean that $F \le_{n \mc P} P$.

Let $\plus : n\mc P \to \mc P$ be the map sending $(P_1,\dots,P_n)$ to $P_1+\dots+P_n$. Then by the definition of $n\mc P$ this is a poset map. Furthermore, $(n\mc P, \plus)$ is a perfect $\mc P$-poset.

Let $X$ be an $n\mc P$-complex. Define the \emph{$n$-support} of $X$ to be
\[
\abs{X}_n := (Q_1,\dots,Q_n)
\]
where
\[
Q_k = \bigcup_{(P_1,\dots,P_n) \in X} P_k.
\]
It is well-known \cite{San05} that if $\abs{X}$ is a polytope, then $Q_1$, \dots, $Q_n$ are polytopes and $\abs{X} = Q_1 + \dots + Q_n$. The set $X$ is known as a \emph{mixed subdivision} of $(Q_1,\dots,Q_n)$. In addition, the map $\abs{\cdot}_n : \Subd(n\mc P) \to n\mc P$ which sends $X$ to $\abs{X}_n$ is a poset map over $\mc P$. In particular, it is a poset map, so $(\Subd(n\mc P), \abs{\cdot}_n)$ is an $n\mc P$-poset.

Now suppose $(\mc A, p)$ is an $n\mc P$-poset. Then $(\mc A, \plus \circ p)$ is a $\mc P$-poset. For the rest of the paper, whenever we define an $n \mc P$-poset $(\mc A, p)$, we will also implicitly treat $\mc A$ as a $\mc P$-poset as above. In particular, we can define $\Subd(\mc A)$ as the poset of $(\mc A, \plus \circ p)$-subdivisions.

Let $(\mc A, p)$ be an $n\mc P$-poset, and let $X$ be an $\mc A$-complex. Then $p(X)$ is an $n\mc P$ complex.  We call $\abs{p(X)}_n$ the \emph{$n$-support} of $X$, and for convenience we denote it by $\abs{X}_n$. Suppose additionally that $X$ is an $\mc A$-subdivision.
By Proposition~\ref{prop:mapofsubdivisions}, the map $p : \Subd(\mc A) \to \Subd(n\mc P)$ given by $Y \mapsto p(Y)$, is a poset map over $\mc P$. As noted previously, $\abs{\cdot}_n : \Subd(n\mc P) \to n\mc P$ is also a poset map over $\mc P$. Hence, the map $\abs{\cdot}_n : \Subd(\mc A) \to n\mc P$ given by $X \mapsto \abs{p(X)}_n = \abs{X}_n$ is a poset map over $\mc P$. It follows that $\abs{X} = \Sum \abs{X}_n$. In addition, $(\Subd(\mc A), \abs{\cdot}_n)$ is an $n\mc P$-poset.

We note the following fact for later.

\begin{prop}\label{prop:nsupporttrivial}
Let $(\mc A,p)$ be an $n \mc P$-poset and let $X$ be an $\mc A$-subdivision. If $X$ has a single maximal element $x$, then $\abs{X}_n = p(x)$.
\end{prop}

\begin{proof}
If $X$ has a single maximal element $x$, then $p(X)$ has a single maximal element $p(x)$. Therefore,
\begin{align*}
\abs{X}_n &= \left( \bigcup_{(P_1,\dots,P_n)\in p(X)} P_k \right)_{k=1}^n \\
&= \left( p(x)_k \right)_{k=1}^n \\
&= p(x). 
\end{align*}
\end{proof}

\subsection{Canonical subdivisions}

Let $(\mc A, p)$, $(\mc A', p')$ be $n\mc P$-posets. We define a \emph{canonical subdivsion} over $n \mc P$ to be any poset map over $n \mc P$ from $(\mc A,p)$ to $(\Subd(\mc A'),\abs{\cdot}_n)$. The importance of these maps is the following proposition.

\begin{prop} \label{prop:canonicalrefinement}
Let $\Sigma : \mc A \to \Subd(\mc A')$ be a canonical subdivision over $n\mc P$. For $X$ an $\mc A$-complex, define
\[
\Sigma^\ast(X) := \bigcup_{x \in X} \Sigma(x).
\]
Then $\Sigma^\ast(X)$ is an $\mc A'$-complex with $n$-support $\abs{X}_n$. Moreover, the map $\Sigma^\ast : \Subd(\mc A) \to \Subd(\mc A')$ given by $X \mapsto \Sigma^\ast(X)$ is a canonical subdivision over $n\mc P$.
\end{prop}

Before proving this, we set some notation. Let $(\mc A, p)$, $(\mc A', p')$ be $n\mc P$-posets. Let $\mc A_0$ be any subset of $\mc A$. A \emph{subdivision} over $n\mc P$ is any map of sets $\Sigma : \mc A_0 \to \Subd(\mc A')$ such that $p(x) = \abs{\Sigma(x)}_n$ for all $x \in \mc A_0$. A canonical subdivision is therefore a subdivision $\Sigma : \mc A \to \Subd(\mc A')$ which is a poset map.

Let $\Sigma : \mc A_0 \to \Subd(\mc A')$ be a subdivision over $n\mc P$, and let $x \in \mc A_0$ and $y \in \mc A$ with $y \le x$. The \emph{restriction} of $\Sigma$ from $x$ to $y$, denoted by $\Sigma(x|y)$, is the unique $Y \in \Subd(\mc A')$ such that $Y \le \Sigma(x)$ and $\abs{Y}_n = p(y)$. Equivalently, this is the unique $Y \in \Subd(\mc A')$ such that $Y \le \Sigma(x)$ and $\abs{Y} = (\Sum \circ p)(y)$. We note the following:

\begin{prop} \label{prop:canonicalcriterion}
A subdivision $\Sigma : \mc A \to \Subd(\mc A')$ over $n\mc P$ is a poset map (and hence a canonical subdivision over $n\mc P$) if and only if $\Sigma(x|y) = \Sigma(y)$ for all $x$, $y \in \mc A$ such that $y \le x$.
\end{prop}

\begin{proof}
If $\Sigma$ is a poset map and $x$, $y \in \mc A$ such that $y \le x$, then $\Sigma(y) \le \Sigma(x)$. Since $\abs{\Sigma(y)}_n = p(y)$ by the definition of a subdivision over $n\mc P$, we must have $\Sigma(x|y) = \Sigma(y)$ by the definition of $\Sigma(x|y)$. Conversely, if $\Sigma(x|y) = \Sigma(y)$ for all $x$, $y \in \mc A$ such that $y \le x$, then in particular $\Sigma(y) = \Sigma(x|y) \le \Sigma(x)$, so $\Sigma$ is a poset map.
\end{proof}

\begin{proof}[Proof of Proposition~\ref{prop:canonicalrefinement}]

Let $x$, $y \in X$ such that $(\Sum \circ p)(x)$ and $(\Sum \circ p)(y)$ share a face. Let $z$ be the element of $X$ such that $p(z)$ is this face, so $z \le x$ and $z \le y$. Since $\Sigma$ is a canonical subdivision, by Proposition~\ref{prop:canonicalcriterion} we have $\Sigma(x|z) = \Sigma(y|z) = \Sigma(z)$. Thus, the polytopal complexes $(\Sum \circ p')(\Sigma(x))$ and $(\Sum \circ p')(\Sigma(y))$ have the subdivision $(\Sum \circ p')(\Sigma(z))$ as their common intersection. This implies that $\Sigma^\ast(X)$ is indeed an $\mc A'$-complex. %Since $\abs{\Sigma^\ast(X)} = \abs{X}$, $\Sigma^\ast(X)$ is an $\mc A'$-subdivision, as desired.

We next show that $\abs{\Sigma^\ast(X)}_n = \abs{X}_n$. We have
\[
\abs{\Sigma^\ast(X)}_n = (Q_1,\dots,Q_n)
\]
where
\begin{align*}
Q_k &= \bigcup_{(P_1,\dots,P_n) \in p ((\Sigma^\ast(X))} P_k \\
&= \bigcup_{x \in X} \bigcup_{(P_1,\dots,P_n) \in p (\Sigma(x))} P_k \\
&= \bigcup_{x \in X} (k^\text{th} \text{ entry of } \abs{\Sigma(x)}_n) \\
&= \bigcup_{x \in X} (k^\text{th} \text{ entry of } p(x)) \\
&= \bigcup_{(P_1,\dots,P_n) \in p(X)} P_k \\
&= k^\text{th} \text{ entry of } \abs{X}_n
\end{align*}
where we have used in the fourth line that $p(x) = \abs{\Sigma(x)}_n$ since $\Sigma$ is a poset map over $n\mc P$. Hence $\abs{\Sigma^\ast(X)}_n = \abs{X}_n$, as desired.

It follows that we have a map $\Sigma^\ast : \Subd(\mc A) \to \Subd(\mc A')$ that is a subdivision over $n \mc P$. To show that it is canonical, we need to show it is a poset map. Let $X$, $Y \in \Subd(\mc A)$ such that $Y \le X$. Thus $\abs{Y}$ is a face of $\abs{X}$, and hence $\abs{\Sigma^\ast(Y)}$ is a face of $\abs{\Sigma^\ast(X)}$ since $\abs{\Sigma^\ast(X)} = \abs{X}$ for all $X \in \Subd(\mc A)$. Moreover, since $Y$ is a subcomplex of $X$, it is easy to see from the definition of $\Sigma$ that $\Sigma(Y)$ is a subcomplex of $\Sigma(X)$. Thus $\Sigma^\ast(Y) \le \Sigma^\ast(X)$, as desired.
\end{proof}

\begin{exam}
For any perfect $n \mc P$-poset $(\mc A,p)$, recall the trivial subdivision $\triv_{\mc A} : \mc A \to \Subd(\mc A)$ given by $\triv_{\mc A}(x) = \langle x \rangle_{\mc A}$. By Proposition~\ref{prop:nsupporttrivial}, we have $\abs{\triv_{\mc A}(x)}_n = p(x)$, hence $\triv_{\mc A}$ is a subdivision over $n \mc P$. It is also clear that if $x$, $y \in \mc A$ with $y \le x$ then $\triv_{\mc A}(x|y) = \triv_{\mc A}(y)$, so $\triv_{\mc A}$ is a canonical subdivision over $n \mc P$.
\end{exam}

\subsection{Confluent subdivisions} \label{sec:confluent}

In this paper we will construct very complicated subdivisions recursively from smaller, simpler subdivisions. The purpose of this section is to give a systematic way to prove that the final subdivisions are well-defined and canonical. The key idea is the notion of confluence and Newman's lemma.

Let $(\mc A, p)$ be an $n\mc P$-poset. Let $\{\mc A_\alpha\}$ be a (possibly infinite) collection of subsets of $\mc A$. For each $\alpha$, let $\sigma_\alpha : \mc A_\alpha \to \mc \Subd(\mc A)$ be a subdivision over $n\mc P$.

Let $X$ be any finite subset of $\mc A$. Suppose $x$ is an element of $X$ and $x \in \mc A_\alpha$ for some $\alpha$. We define a \emph{$\sigma_\alpha$-move} on $X$ at $x$ to be the act of transforming $X$ into the set
\[
X' := X \setminus \{x\} \cup \Max_{\mc A}\sigma_\alpha(x).
\]
We write this move as $X \xrightarrow{x,\sigma_\alpha} X'$, or simply $X \to X'$ if we do not need to specify $(x,\sigma_\alpha)$. We call a move $X \to X'$ \emph{non-trivial} if $X \neq X'$.

Given another finite set $Y \subset \mc A$, we write $X \xrightarrow{\ast} Y$ if there exists a sequence of moves $X \to X_1 \to X_2 \to \dots \to Y$. We allow this sequence to contain only one term; in other words, we always have $X \xrightarrow{\ast} X$.

Suppose that $X$, $Y \subset \mc A$ are finite. We say that $X$ and $Y$ are \emph{joinable} if there exists $Z \subset \mc A$ such that $X \xrightarrow{\ast} Z$ and $Y \xrightarrow{\ast} Z$.

We say that the family of subdivisions $\{\sigma_\alpha\}$ is \emph{locally confluent} if for any finite $X \subset \mc A$ and any two moves $X \to Y_1$, $X \to Y_2$, we have that $Y_1$ and $Y_2$ are joinable. Note that this is equivalent to saying that for all $x \in \mc A$ and $\alpha$, $\beta$ such that $x \in \mc A_\alpha \cap \mc A_\beta$, we have that $\Max \sigma_\alpha(x)$ and $\Max \sigma_\beta(y)$ are joinable.

We call a set $X \subset \mc A$ \emph{terminal} if there are no non-trivial moves from $X$. We call an element $x \in \mc A$ \emph{terminal} if $\{x\}$ is terminal. An element $x$ is terminal if and only if for every $\alpha$ such that $x \in \mc A_\alpha$, we have $x \in \sigma_\alpha(x)$. (This is because $\sigma_\alpha(x)$ is an $\mc A$-subdivision of $(\Sum \circ p)(x)$, so if $x \in \sigma_\alpha(x)$, then $x$ must be the only maximal element of $\sigma_\alpha(x)$.) From this, it follows that a set is terminal if and only if all its elements are terminal.
 
We say that $\{\sigma_\alpha\}$ is \emph{terminating} if there are is no infinite sequence $X_1 \to X_2 \to X_3 \to \dots$ of non-trivial moves.

Finally, we say that $\{\sigma_\alpha\}$ is \emph{facially compatible} if the following two properties hold:
\begin{itemize}
\item If $x$, $y \in \mc A$ such that $y \le x$ and $x \in \mc A_\alpha$, then $\Max\sigma_\alpha(x|y)$, $\{y\}$ are joinable.
\item If $x$, $y \in \mc A$ such that $y \le x$ and $x$ is terminal, then $y$ is terminal.
\end{itemize}

Our main result is the following.

\begin{thm} \label{thm:confluence}
Let $(\mc A, p)$ be a perfect $n\mc P$-poset and let $\{\sigma_\alpha : \mc A_\alpha \to \Subd(\mc A)\}$ be a family of locally confluent, facially compatible, and terminating subdivisions. Then for any $x \in \mc A$, there is a unique terminal set $S(x) \subset \mc A$ such that $\{x\} \xrightarrow{\ast} S(x)$. Moreover, $\Sigma(x) := \langle S(x) \rangle_{\mc A}$ is an $\mc A$-subdivision, and the map $\Sigma : \mc A \to \Subd(\mc A)$ given by $x \mapsto \Sigma(x)$ is a canonical subdivision over $n \mc P$.
\end{thm}

\begin{proof}
The fact that $S(x)$ exists and is unique follows directly from Newman's diamond lemma, which states that any locally confluent and terminating binary relation is globally confluent \cite{New42}. (In this case, the binary relation is non-trivial moves ``$\to$''.)

It remains to show that $\Sigma(x)$ is an $\mc A$-subdivision, and the map $\Sigma : \mc A \to \Subd(\mc A)$ is a canonical subdivision over $n \mc P$. Define a binary relation $\prec$ on $\mc A$ by $y \prec x$ if there is a non-trivial move $\{x\} \to Y$ such that $y \in Y$. We prove the following:

{\bf Claim:} $\prec$ is a well-founded relation on $\mc A$.

\begin{proof}[Proof of claim]
Suppose that $x_1 \succ x_2 \succ x_3 \succ \dots$ is an infinite descending sequence. Define a sequence of non-trivial moves $X_1 \to X_2 \to \dots$, inductively as follows. Define $X_1 = \{x_1\}$. Next fix $k \ge 2$, and assume by induction that we have constructed $X_1 \to \dots \to X_{k-1}$ such that $x_j \in X_j$ for all $1 \le j \le k-1$. By definition of $\prec$, there exists a non-trivial move $x_{k-1} \xrightarrow{(x_{k-1},\sigma_\alpha)} Y_k$ such that $x_k \in Y_k$. Define $X_k$ by $X_{k-1} \xrightarrow{(x_{k-1},\sigma_\alpha)} X_k = X_{k-1} \setminus \{x_{k-1}\} \cup Y_k$. Since the move $x_{k-1} \to Y_k$ is non-trivial, we have $x_{k-1} \notin Y_k$. Thus $x_{k-1} \notin X_k$, so $X_{k-1} \neq X_k$. Hence we have a non-trivial move $X_{k-1} \to X_k$ with $x_k \in X_k$, completing the induction. However, this contradicts the terminating property, proving the claim.
\end{proof}

We now prove by induction on $\prec$ that
\begin{enumerate}[label=(\alph*)]
\item $\Sigma(x)$ is an $\mc A$-subdivision
\item $\abs{\Sigma(x)}_n = p(x)$
\item $\Sigma(x|y) = \Sigma(y)$ for all $y \le x$.
\end{enumerate}
By Proposition~\ref{prop:canonicalcriterion}, this will complete the proof.

For the base case, assume that $x$ is terminal. Thus $S(x) = \{x\}$, so $\Sigma(x) = \langle x \rangle_{\mc A} = \triv_{\mc A}(x)$, since $\mc A$ is perfect. Hence $\Sigma(x)$ is an $\mc A$-subdivision, proving (a).  Proposition~\ref{prop:nsupporttrivial} implies (b). Suppose $y \le x$. By facial compatibility, $y$ is also terminal, so $\Sigma(y) = \triv_{\mc A}(y)$. Since $\triv_{\mc A}(x|y) = \triv_{\mc A}(y)$, (c) is proved.

For the inductive step, assume that $x$ is not terminal. Let $\{x\} \xrightarrow{(x,\sigma_\alpha)} Y$ be a non-trivial move. In particular, we have $x \notin Y$. Then $S(x) = \bigcup_{y \in Y} S(y)$, and hence $\Sigma(x) = \bigcup_{y \in Y} \Sigma(y)$. By induction, we have proved (a)-(c) for all $y \in Y$. Thus, the proof of Proposition~\ref{prop:canonicalrefinement} implies (a) and (b) for $x$. For (c), let $y \le x$. By facial compatibility, we have that $\Max\sigma_\alpha(x|y)$ and $\{y\}$ are joinable. It follows that
\[
S(y) = \bigcup_{z \in \Max\sigma_\alpha(x|y)} S(z)
\]
and hence
\begin{equation} \label{eq:Sigmay}
\Sigma(y) = \bigcup_{z \in \Max\sigma_\alpha(x|y)} \Sigma(z).
\end{equation}
For every $z \in \Max\sigma_\alpha(x|y)$, there is $z' \in \Max \sigma_\alpha(x) = Y$ such that $z \le z'$. Since $z' \neq x$ by assumption that $\{x\} \to Y$ is non-trivial, we have $z' \prec x$, and therefore by induction $\Sigma(z) = \Sigma(z'|z)$. Hence $\Sigma(y)$ is a union of elements of the form $\Sigma(z'|z)$ with $z' \in Y$. Since $\Sigma(x) = \bigcup_{z' \in Y} \Sigma(z')$ and $\Sigma(z'|z)$ is a subcomplex of $\Sigma(z')$, it follows that $\Sigma(y)$ is a subcomplex of $\Sigma(x)$. On the other hand, the formula \eqref{eq:Sigmay} implies the support of $\Sigma(y)$ is $\abs{y}$. So we must have $\Sigma(x|y) = \Sigma(y)$, completing the proof.
\end{proof}

\subsection{Examples of canonical subdivisions} \label{sec:canonicalexamples}

In this section we give two examples of well-known subdivisions which can be expressed as canonical subdivisions in our sense. This section can be skipped without logically affecting the main proof, but the ideas may be useful for understanding the later arguments.

\subsubsection{Pulling triangulations}

Let $(\mc S, \conv)$ be the $\mc P$-poset from Example~\ref{exam:pointsets}. Let $A \in \mc S$. A \emph{covector} of $A$ is a point $x \in A$ such that $\dim(A \setminus \{x\}) < \dim(A)$. (Here, $\dim(A)$ denotes the dimension of the smallest affine subspace containing $A$.) Every element of $A$ is a covector if and only if $A$ is affinely independent, i.e. $A$ is the set of vertices of a simplex.

Let $\mc S^\ast$ be the set of elements of $\mc S$ which are not affinely independent. We define a subdivision $\pull : \mc S^\ast \to \Subd(\mc S)$ as follows. Let $A \in \mc S^\ast$, and let $x$ be the smallest element of $A$ (according to the order on $A$) which is not a covector. We define $\pull(A)$ to be the set of all $B$ and $x \cup B$ such that $B \le_{\mc S} A$ and $x \notin B$. Then $\pull(A)$ is an $\mc S$-subdivision of $\abs{A}$, so we have a subdivision $\pull : \mc S^\ast \to \Subd(\mc S)$ over $\mc P$.

We now claim that the family $\{\pull\}$ consisting of a single subdivision is locally confluent, terminating, and facially compatible. Local confluence is trivial since the family has only one subdivision. If $A \in \mc S^\ast$ and $B \in \Max \pull(A)$, then $B \subsetneq A$, which proves termination. Finally, suppose $A \in \mc S^\ast$ and $B \le_{\mc S} A$. If $x \notin B$, then $\pull(A|B) = \triv_{\mc S}(B)$. If $x \in B$, then we can check that
\[
\pull(A|B) = \begin{dcases*}
\pull(B) & if $x$ is not a covector of $B$ \\
\triv_{\mc S}(B) & otherwise.
\end{dcases*}
\]
In all cases, $\Max\pull(A|B)$ and $\{B\}$ are joinable, proving the first condition of facial compatibility. For the second property, we note that $A$ is terminal if and only if $A \notin \mc S^\ast$, i.e. $A$ is affinely independent. Clearly if this holds for $A$ then it holds for $B$, completing the proof.

Thus, by Theorem~\ref{thm:confluence}, we have a canonical subdivision $\Pull : \mc S \to \Subd(\mc S)$ such that if $A \in \mc S$ and $B \in \Pull(A)$, then $B$ is terminal, i.e. affinely independent. This subdivision is known in the literature as the \emph{pulling triangulation}.

\subsubsection{Dicing}

Let $\ms H$ be a finite set of hyperplanes in $\bb R^d$. For each $H \in \ms H$, let $\mc P_H$ be set of polytopes in $\bb R^d$ whose relative interior intersects $H$. We define a subdivision $\dice_H : \mc P_H \to \Subd(\mc P)$ as follows. Let $H_1$, $H_2$ be the two closed half-spaces of $\bb R^d$ cut out by $H$. For $P \in \mc P_H$, we define
\[
\dice_H(P) = \triv(P \cap H_1) \cup \triv(P \cap H_2)
\]
Then $\dice_H(P)$ is a polytopal subdivision of $P$ and $\dice_H : \mc P_H \to \Subd(\mc P)$ is a subdivision over $\mc P$.

We claim that the family $\{ \dice_H \}_{H \in \ms H}$ is locally confluent, terminating, and facially compatible. We start with terminating. If $A \in \mc P_H$ and $B \in \Max\dice(A)$, then $B \notin \mc P_H$. Since $\ms H$ is finite, this implies $\{ \dice_H \}_{H \in \ms H}$ is terminating. We next show facial compatibility. Let $A \in \mc P_H$ and let $B \le_{\mc P} A$. If $B \in \mc P_H$, then $\dice_H(A|B) = \dice_H(B)$. Otherwise, $\dice_H(A|B) = \triv(B)$. Either way, $\Max \dice_H(A|B)$ and $\{B\}$ are joinable, proving the first condition of facial compatibility. Next, note that $A$ is terminal if and only if its relative interior does not intersect any $H \in \ms H$. If this holds for $A$ then it clearly holds for $B$, proving facial compatibility.

Finally, we prove local confluence. Suppose $P \in \mc P_{G} \cap \mc P_{H}$ for distinct $G$, $H \in \ms H$. Consider $\Max \dice_G(P)$ and $\Max \dice_H(P)$. If we apply a $\dice_H$-move to both elements of $\Max \dice_G(P)$, then we obtain the same result as when we apply $\dice_G$-move to both elements of $\Max \dice_H(P)$. Thus $\Max \dice_G(P)$ and $\Max \dice_H(P)$ are joinable, so $\{ \dice_H \}_{H \in \ms H}$ is locally confluent.

Thus, by Theorem~\ref{thm:confluence}, we have a canonical subdivision $\Dice : \mc P \to \Subd(\mc P)$ such that for all $P \in \mc P$ and $Q \in \Dice(P)$, the relative interior of $Q$ does not intersect any $H \in \ms H$. This is of course the subdivision obtained by intersecting $P$ with each of the closed regions of $\bb R^d$ cut out by $\ms H$.

\section{Cayley polytopes} \label{sec:Cayleypolytopes}

\subsection{Notation}

Before proceeding, we set some notation regarding tuples and matrices. Let $a = (a_1, \dots, a_m)$ be an $m$-tuple (entry type unspecified). We allow $m=0$, in which case the tuple has no entries. We define $\abs{a} := m$. For $I \subset [m]$, we use $a_I$ to denote the tuple $(a_{i_1}, \dots, a_{i_k})$ where $I = \{i_1,\dots,i_k\}$ and $i_1 < \dots < i_k$. We use $a_{\setminus j}$ to denote $a_{[m] \setminus \{j\}}$.

Similarly, let $a = (a_{ij})_{i,j=1}^{m,n}$ be an $m \times n$ matrix and $I \subset [m]$, $J \subset [n]$. We allow $m=0$ or $n=0$, in which case the matrix has no entries. We use $a_{I \times \bullet}$ to denote the matrix obtained by restricting $a$ to the rows indexed by $I$, preserving the order of the rows. We similarly define $a_{\bullet \times J}$ and $a_{I \times J}$. We use $a_{\setminus i \times \bullet}$ to denote $a_{I \setminus \{i\} \times \bullet}$, and similarly define $a_{\bullet \times \setminus j}$ and $a_{\setminus i \times \setminus j}$.

\subsection{Cayley sums} \label{sec:Cayley}

%Given a two tuples of polytopes $(P_1,\dots,P_n)$ and $(Q_1,\dots, Q_n)$, we write $P \le Q$ if $P_j$ is a face of $Q_j$ for all $j$.

Let $P = (P_1,\dots,P_m)$ be an $m$-tuple of polytopes in $\bb R^d$, where $m \ge 1$. We say that $P$ is in \emph{Cayley position} if there exists a linear map $\pi: \bb R^d \to \bb R^d$ such that $\pi(P_i)$ is a point for all $i$ and the sequence of points $(\pi(P_i))_{i \in [m]}$ is affinely independent. In this situation, we define the \emph{Cayley sum} $\Cay(P)$ to be the convex hull of the entries of $P$.

The faces of a Cayley sum $\Cay(P)$ are precisely Cayley sums of the form $\Cay(F_I)$, where $I$ is a nonempty subset of $[m]$ and $F = (F_1, \dots, F_m)$ where $F \le P$. (Recall that $F \le P$ means that there exists a linear functional such that $F_1$, \dots, $F_m$ are the faces of $P_1$, \dots, $P_m$, respectively, with respect to this linear functional.)

Let $S = (S_1,\dots,S_n)$ be an $n$-tuple of independent simplices (recall the definition of independence from Section~\ref{sec:polytopes}), and let $a = (a_{ij})_{i,j=1}^{m,n}$ be an $m \times n$ matrix of nonnegative integers, where $m \ge 1$ and $n \ge 0$. We consider Cayley sums of the form $\Cay(P)$, where $P = (P_1,\dots,P_m)$ is an $m$-tuple of polytopes in Cayley position, and for each $i \in [m]$ we have
\[
P_i = p_i + \sum_{j = 1}^n a_{ij} S_j
\]
for some $p_i \in \bb R^d$. The faces of $\Cay(P)$ are as follows. Let $S' = (S'_1,\dots,S'_n)$ be any $n$-tuple of polytopes such that $S'_j \le S_j$ for all $j$. Since the entries of $S$ are independent, this is equivalent to $S' \le S$. Let $P' = (P_1',\dots,P_m')$ be the tuple with
\[
P_i' = p_i + \sum_{j=1}^n a_{ij} S_j'.
\]
Then $P' \le P$. Thus, for any nonempty $I \subset [m]$, $\Cay(P'_I)$ is a face of $\Cay(P)$. All faces of $\Cay(P)$ arise this way.

\subsection{The poset $\mc C$}

Let $\mc C_0$ be the poset defined as follows. Its elements are all tuples $(p,S,a)$
where
\begin{itemize}
\item $p = (p_1,\dots,p_m)$ is a tuple of points in $\bb R^d$, for some positive integer $m$.
\item $S = (S_1,\dots,S_n)$ is a tuple of independent ordered integral simplices in $\bb R^d$, for some nonnegative integer $n$.
\item $a = (a_{ij})_{i,j=1}^{m,n}$ is an $m \times n$ matrix of nonnegative integers.
\item The polytopes $P_1$, \dots, $P_m$ are in Cayley position, where
\[
P_i := p_i + \sum_{j = 1}^n a_{ij} S_j.
\]
\end{itemize}
We equip $\mc C_0$ with the partial order $\le _{\mc C_0}$, where
\[
(p_I,S',a_{I \times \bullet}) \le_{\mc C_0} (p,S,a)
\]
if $I$ is a nonempty subset of $[\abs{p}]$ and $S' \le S$. It is easy to see that this is a partial order.

Let $\Cay : \mc C_0 \to \mc P$ be the map defined by
\[
\Cay(p,S,a) = \Cay \left(p_i + \sum_{j = 1}^n a_{ij} S_j \right)_{i \in [\abs{p}]}
\]
By the discussion from the previous subsection, this is a poset map. Thus $(\mc C_0, \Cay)$ is a $\mc P$-poset.

We now define an equivalence relation $\sim$ on $\mc C_0$ as follows. If $j \in [\abs{S}]$ is such that $S_j$ is a point or $a_{ij} = 0$ for all $i \in [\abs{p}]$, then we set
\[
(p,S,a) \sim ((p_i+a_{ij}S_j)_{i \in [\abs{p}]},S_{\setminus j},a_{\bullet \times \setminus j}).
\]
We define $\mc C$ to be $\mc C_0 / \sim$.

For $A$, $B \in \mc C$, we let $A \le_{\mc C} B$ if there are representatives $A_0$, $B_0$ of $A$ and $B$, respectively, in $\mc C_0$ such that $A_0 \le_{\mc C_0} B_0$. It is straightforward to check that this defines a partial order on $\mc C$ (for example, using standard form defined in the next subsection). Furthermore, we have that $\Cay : \mc C_0 \to \mc P$ is constant on equivalence classes of $\sim$. Thus there is a well-defined map $\Cay : \mc C \to \mc P$, this is a poset map, and $(\mc C, \Cay)$ is a $\mc P$-poset.

If there is no risk of confusion, we will abuse notation and denote elements in $\mc C$ using their representatives in $\mc C_0$.

\subsubsection{Standard form and $L(A)$}

Given an object $A \in \mc C$, there is a unique representative $(p,S,a)$ of $A$ in $\mc C_0$ such that for all $j$, the $j$-th column of $a$ is not all zeroes. We call this the \emph{standard form} of $A$. It is easy to check that if $A_0$ is the standard form of $A \in \mc C$, then $B \le_{\mc C} A$ if and only if $B$ has a representative $B_0 \in \mc C_0$ such that $B_0 \le_{\mc C_0} A_0$. (In fact, this is true if $A_0$ is any representative of $A$, which can be proved from the previous sentence.)

%If $A \in \mc C$ has standard form $(p,S,a)$, then we define $m(A) := \abs{p}$, and $n(A) := \abs{S}$.

Using standard form, it is not hard to see that if $A \in \mc C$ and we have two different elements $B$, $C \in \mc C$ such that $B \le_{\mc C} A$ and $C \le_{\mc C} A$, then $\Cay(B)$ and $\Cay(C)$ are different faces of $\Cay(A)$. Moreover, for every face $F$ of $\Cay(A)$, there is a face $B$ of $A$ such that $\Cay(B) = F$. It follows that $\mc C$ is a perfect $\mc P$-poset.

Let $A \in \mc C$ with standard form $(p,S,a)$.
For $j \in [\abs{S}]$, let $v_j$ be the first vertex of $S_j$. Define
\[
S_0(A) := \conv \left( p_i + \sum_{j=1}^{\abs{S}} a_{ij} v_j \right)_{i \in [\abs{p}]}.
\]
In other words, $S_0(A)$ is the face of $\Cay(A)$ whose vertices are the first vertices of each Cayley summand of $\Cay(A)$.

Note that $S_0(A)$, $S_1$, \dots, $S_{\abs{S}}$ are independent simplices. We define
\begin{align*}
L(A) &:= L(S_0(A) + S_1 + \dots + S_{\abs{S}}) \\
&= L(S_0(A)) \oplus L(S_1) \oplus \dots \oplus L(S_{\abs{S}}).
\end{align*}

\subsection{$\gamma_T$ subdivisions} \label{sec:T}

\subsubsection{Definition of $\gamma_T$}

Let $T$ be an ordered integral simplex of dimension at least 1. Let $\mc C_T$ be the set of elements of $\mc C$ whose standard form $(p,S,a)$ has the property that $T$ is an entry of $S$. In this section we construct a subdivision $\gamma_T : \mc C_T \to \Subd(\mc C)$ over $\mc P$ such that the family $\{\gamma_T\}$ of all such subdivisions satisfies the conditions of Theorem~\ref{thm:confluence}. These subdivisions will be a main building block for future subdivisions.

Let $A \in \mc C_T$ with standard form $(p,S,a)$. Let $j$ be the unique number such that $S_j = T$, and let $i$ be the smallest number such that $a_{ij} = \max_{i'j} a_{i'j}$. Let $v$ be the first vertex of $T$, and let $f$ be the facet of $T$ opposite $v$.

Note that for any positive integer $c$, the dilated simplex $cT$ can be written as the union of the two polytopes
\[
v + (c-1)T, \quad \Cay((c-1)f, cf).
\]
(Here, we are using the ordinary Cayley sum as defined in Section~\ref{sec:Cayley}.) This can be extended to a $\mc C$-subdivision of $\Cay(A)$ as follows. Recall that the standard form of $A$ is $(p,S,a)$. Let $m := \abs{p}$ and $n := \abs{S}$.
We define
\begin{equation} \label{eq:A'A''}
\begin{aligned}
A' &= (p',S,a') \in \mc C \\
A'' &= (p'', F, a'') \in \mc C
\end{aligned}
\end{equation}
where
\begin{itemize}
\item $p'$ is the $m$-tuple obtained by replacing the $i$-th entry of $p$ with $p_{i} + v$.
\item $a'$ is the $m \times n$ matrix obtained by subtracting 1 from the $(i,j)$ entry of $a$.
\item $p''$ is the $(m+1)$-tuple obtained by inserting $p_{i} + v$ directly before the $i$-th entry of $p$.
\item $F$ is the $n$-tuple obtained by replacing the $j$-th entry of $S$ with $f$.
\item $a''$ is the $(m+1) \times n$ matrix obtained by inserting the $i$-th row of $a'$ directly above the $i$-th row of $a$.
\end{itemize}

We define
\[
\gamma_T(A) := \triv_{\mc C}(A') \cup \triv_{\mc C}(A'').
\]
Then $\gamma_T(A)$ is a $\mc C$-subdivision of $\Cay(A)$. If $a_{ij} = 1$ and $a_{i'j} = 0$ for all $i' \neq i$, then this subdivision has a unique maximal element $A''$. Otherwise, the maximal elements are $A'$ and $A''$.

\subsubsection{Properties of $\gamma_T$}

We first prove the following:

\begin{prop} \label{prop:LgammaT}
Let $A \in \mc C_T$ and $B \in \Max \gamma_T(A)$. Then $L(A) = L(B)$.
\end{prop}

\begin{proof}
We have either $B = A'$ or $B = A''$. In the former case, we have $S_0(A) = S_0(B)$ and $A$ and $B$ have the same second entry, so $L(A) = L(B)$ as desired.

Assume $B = A''$. Without loss of generality, assume $i = j = 1$. For convenience, we will reuse the variables $i$ and $j$ in the below proof. Let $w$ be the second vertex of $T$. We have
\begin{align*}
S_0(B) &= \conv \left( \left\{ p_i + a_{i1} w + \sum_{j=2}^{\abs{S}} a_{ij} v_j \right\}_{i \in [\abs{p}]} \cup \right. \\
&\qquad\qquad \left. \left\{ p_{1} + v + (a_{11}-1)w + \sum_{j=2}^{\abs{S}} a_{1j}v_j \right\}
\right) \\
&= \conv \left( \left\{ p_i + a_{i1} (w-v) + \sum_{j=1}^{\abs{S}} a_{ij} v_j \right\}_{i \in [\abs{p}]} \cup \right. \\
&\qquad\qquad \left. \left\{ p_{1} + (a_{11}-1)(w-v) + \sum_{j=1}^{\abs{S}} a_{1j}v_j \right\}
\right).
\end{align*}
Subtracting the last entry in the above convex hull from the first entry, we get $w-v$. Hence $w-v \in L(S_0(B))$. It is then not hard to see from the above expression that $L(S_0(B)) = L(S_0(A)) \oplus \bb Z\langle w-v \rangle$. Thus,
\begin{align*}
L(B) &= L(S_0(A)) \oplus \bb Z\langle w-v \rangle \oplus L(f) \oplus L(S_2) \oplus \dots \oplus L(S_{\abs{S}}) \\
&= L(S_0(A)) \oplus L(T) \oplus L(S_2) \oplus \dots \oplus L(S_{\abs{S}}) \\
&= L(A)
\end{align*}
as desired.
\end{proof}

\begin{prop} \label{prop:facegammaT}
Let $A \in \mc C_T$ and suppose $\mc B \in \mc C$ such that $B \le_{\mc C} A$. Then either $\gamma_T(A|B) = \gamma_{T'}(B)$ for some $T'$ or $\gamma_T(A|B) = \triv_{\mc C}(B)$.
\end{prop}

\begin{proof}
Let the standard form of $A$ be $(p,S,a)$. Let $B = (p_I,S',a_{I \times \bullet})$ with $S' \le S$ and $I$ a nonempty subset of $[\abs{p}]$. Assume that $T = S_j$, and let $i$ be the smallest number such that $a_{ij} = \max_{i'} a_{i'j}$. If $i \in I$ and $S_j'$ contains the first vertex of $T$ and another vertex, then $\gamma_T(A|B) = \gamma_{S_j'}(B)$. Otherwise, we have $\gamma_T(A|B) = \triv_{\mc C}(B)$.
\end{proof}

We now prove the following.

\begin{prop} \label{prop:confgammaT}
The family of subdivisions $\{ \gamma_T \}$, where $T$ ranges over all ordered integral simplices of dimension at least 1, is locally confluent, terminating, and facially compatible.
\end{prop}

\begin{proof}
We first prove local confluence. Let $A \in \mc C$ with standard form $(p,S,a)$, and let $T_1$, $T_2$ be distinct entries of $S$. We need to show that $\Max \gamma_{T_1}(A)$ and $\Max \gamma_{T_2}(A)$ are joinable. Without loss of generality, assume $T_1 = S_1$ and $T_2 = S_2$.

Let $i_1$, $i_2$ be the smallest numbers such that $a_{i_1 1} = \max_i a_{i1}$ and $a_{i_22} = \max_i a_{i2}$, respectively. First suppose that $i_1 \neq i_2$. Then applying a $\gamma_{T_1}$-move on $\Max \gamma_{T_2}(A)$ at each of its elements yields the same result as applying a $\gamma_{T_2}$-move on $\Max \gamma_{T_1}(A)$ at each of its elements. Hence $\Max \gamma_{T_1}(A)$ and $\Max \gamma_{T_2}(A)$ are joinable, as desired.

Now assume $i_1 = i_2$. Let $A_1'$, $A_1'' \in \gamma_{T_1}(A)$ be as in \eqref{eq:A'A''} and define $A_2'$, $A_2'' \in \gamma_{T_2}(A)$ analogously. Consider the following sequence of moves starting from $\gamma_{T_1}(A)$. First, if $A_1' \in \Max(\gamma_{T_1}(A))$, then apply a $\gamma_{T_2}$ move at $A_1'$. Next, apply a $\gamma_{T_2}$ move at $A_1''$. Since $A_1''$ has at least two entries greater than 0 in the second column, $\gamma_{T_2}(A_1'')$ has two maximal elements $(A_1'')'$, $(A_1'')''$, defined analogously to \eqref{eq:A'A''}. Finally, apply a $\gamma_{T_2}$ move at $(A_1'')'$.

If we do an analogous sequence of moves starting from $\gamma_{T_2}(A)$ (with $\gamma_{T_1}$ moves replacing $\gamma_{T_2}$ moves), then we obtain the same set as we did with the above sequence. Hence, $\gamma_{T_1}(A)$ and $\gamma_{T_2}(A)$ are joinable, as desired.

We next show that $\{ \gamma_T \}$ is terminating. Let $A \in \mc C_T$ with standard form $(p,S,a)$. Then if $B \in \mc C_T$ with standard from $(p',S',a')$, then either $\sum \dim S'_j < \sum \dim S_j$, or $S = S'$ and $\sum_{i,j} a'_{ij} < \sum_{i,j} a_{ij}$. Since the quantities $\sum \dim S_j$ and $\sum_{i,j} a_{ij}$ are both nonnegative integers, it follows that $\{ \gamma_T \}$ is terminating.

Finally, we show that $\{ \gamma_T \}$ is facially compatible. Let $A \in \mc C_T$ and suppose $\mc B \in \mc C$ such that $B \le_{\mc C} A$. Proposition~\ref{prop:facegammaT} implies $\Max \gamma_T(A|B)$ and $\{B\}$ are joinable, as desired.

Next, suppose $A \in \mc C$. From the definition of $\gamma_T$, $A$ is terminal with respect to $\{ \gamma_T \}$ if and only if $A \notin \mc C_T$ for any $T$. Thus, if $(p,S,a)$ is the standard form of $A$, then $A$ is terminal if and only if $\abs{S} = 0$. Clearly if this holds for $A$, it holds for any $B \le A$. Hence $\{ \gamma_T \}$ is facially compatible, completing the proof.
\end{proof}

From the previous two propositions we immediately have the following from Theorem~\ref{thm:confluence}.

\begin{thm} \label{thm:Gamma}
There is a canonical subdivision $\Gamma : \mc C \to \Subd(\mc C)$ over $\mc P$ such that for any $A \in \mc C$ and any $B \in \Max \Gamma(A)$, where $(p,S,a)$ is the standard form of $B$, we have $\abs{S} = 0$ and $L(A) = L(B)$.
\end{thm}

When restricted to elements $(p,S,a)$ of $\mc C$ where $\abs{p} = 1$ and $\abs{S} = 1$, $\Gamma$ is known as the ``canonical triangulation'' of a dilated ordered simplex \cite{HPPS}.

\section{Box points and index-lowering} \label{sec:boxpoints}

\subsection{Box points}

In this section we describe the standard way of reducing the indices of integral polytopes, first defined in \cite{KKMS} as \emph{Waterman points}. We use the terminology of \cite{HPPS} and call them \emph{box points}.

Let $S = (S_1,\dots,S_n)$ be a tuple of $n$ independent integral simplices in $\bb R^d$. Consider the lattices
\begin{align*}
L(S) &:= L(S_1) + \dots + L(S_n) \\
N(S) &:= N(S_1) + \dots + N(S_n).
\end{align*}
We define a \emph{box point} of $S$ to be any nonzero element of the quotient group $G(S) := N(S) / L(S)$. Let $v_1$, \dots, $v_n$ be any vertices of $S_1$, \dots, $S_n$, respectively.  For any box point $x$ of $S$, there is a unique tuple $(c_1,\dots,c_n)$ of nonnegative integers such that the polysimplex $\sum_{j=1}^n c_j(S_j-v_j)$ contains exactly one representative of $x$ in $N(S)$.\footnote{Explicitly, $(c_1,\dots,c_n)$ can be described as follows. For each $1 \le j \le n$, let $\{u_1,\dots,u_{k_j}\}$ be the vertices of $S_j$ other than $v_j$, and define $e_j^i := u_i - v_j$. Then the $e_j^i$ form a basis for $L(S)$. Every $x \in G(S)$ has a unique representative in $N(S)$ of the form $\sum_{i,j} x_j^i e_j^i$, where each $0 \le x_j^i < 1$. Then $c_j = \lceil \sum_i x_j^i \rceil$.}
The tuple $(c_1,\dots,c_n)$ does not depend on the choice of $v_1$, \dots, $v_n$. We denote the tuple $(c_1,\dots,c_n)$ by $c(S,x)$. We have $0 \le c_j \le d$ for all $j$. We define the \emph{support} of $c$ to be $\supp c := \{ j \in [n] : c_j \neq 0 \}$.

Let $S$ be as above, and let $S' \le S$. Then we have a natural inclusion $G(S') \into G(S)$. Thus, any box point of $S'$ can be regarded as a box point of $S$. Suppose $x$ is a box point of $S'$, and thus a box point of $S$. Then $c(S',x) = c(S,x)$.

Now, let $L$ be a $d$-dimensional sublattice of $\bb Z^d$. Let $x \in \bb Z^d / L$ be nonzero. For any tuple $S = (S_1,\dots,S_n)$ of independent integral simplices in $\bb Z^d$, we say that $x$ is a \emph{box point} of $S$ if there is $S' \le S$ such that $L(S') = L \cap N(S')$ and $N(S')$ contains a representative of $x$. In this situation, clearly $x$ can be naturally identified with a unique nonzero element of $G(S')$, and hence it can be identified with a unique nonzero element of $G(S)$.

Let $x$ be as above. If $A \in \mc C$ has standard form $(p,S,a)$, we say that $x$ is a box point of $A$ if $x$ is a box point of $S$.

%If $x$ is a box point of $S$ and the unique representative of $x$ in $\sum c(S,x)_i S_i$ is in the relative interior of this polytope, then we say that $x$ is a \emph{fundamental} box point of $S$. For any $S$ and any box point $x$ of $S$, there is a unique up to translation $T = (T_1,\dots,T_k)$ such that no $T_i$ is a point, $S$ dominates $T$, and $x$ is a box point of $T$.

\subsection{A proof of the KMW theorem} \label{sec:KMW}

We now use box points to give a short proof of the KMW theorem. The fundamental construction is the same as in the original proof \cite{KKMS}, but the argument structure is simplified due to the use of canonical subdivisions. This section is optional and is not needed in the proof of the main theorem of the paper.

Let $S$ be an ordered integral simplex. Let $x$ be a box point of $S$. Let $c$ be the single entry of $c(S,x)$. Let $x_0$ be the unique representative of $x + cv$ in $cS$, where $v$ is the first vertex of $S$. Let $\mf F = \mf F(S)$ be the set of all facets $F$ of $S$ such that $cF$ does not contain $x_0$. Consider the collection of simplices
\[
\{ \conv(x_0,cF) : F \in \mf F \}.
\]
These simplices are the maximal simplices of a triangulation of $cS$, called the \emph{stellar subdivision} of $cS$ at $x_0$.

We can make this stellar subdivision a $\mc C$-subdivision by considering all elements in $\mc C$ with standard form
\begin{equation} \label{eq:starsimplex}
\left( (x_0,0), (F), \begin{pmatrix} 0 \\ c \end{pmatrix} \right)
\end{equation}
where $F \in \mf F$. These elements map via $\Cay$ to the maximal elements of the above stellar subdivision. Thus, they are the maximal elements of a $\mc C$-subdivision whose image under $\Cay$ is the stellar subdivision.

Crucially, for any $F \in \mf F$, the lattice distance $h'$ between $cF$ and $x_0$ is strictly smaller than the lattice distance $h$ between $F$ and the opposite vertex of $S$. Thus, if $A \in \mc C$ is of the form \eqref{eq:starsimplex}, we have
\[
\ind L(A) = h' \ind L(F) < h \ind L(F) = \ind L(S)
\]

Now consider the polytope $NcS$, where $N$ is a positive integer. Consider the collection of polytopes
\begin{multline*}
\{ 
\{ \conv(x_0+(N-1)cF,NcF) :  F \in \mf F \} \\
\cup \{ \conv((N-1)x_0+cF,(N-2)x_0+2cF) :  F \in \mf F \} \cup \dotsb \\
\cup \{ \conv(Nx_0,(N-1)x_0+cF) : F \in \mf F \} \}
\end{multline*}
These are the maximal elements of a polytopal subdivision of $NcS$. We can make this a $\mc C$-subdivision by considering all elements in $\mc C$ with standard form
\begin{equation} \label{eq:concentricpolytope}
\left( \left( rx_0, (r-1)x_0 \right),(F), \begin{pmatrix} (N-r)c \\ (N-r+1)c \end{pmatrix} \right)
\end{equation}
where $F \in \mf F$ and $1 \le r \le N$ is an integer. These are the maximal elements of a $\mc C$-subdivision whose image under $\Cay$ is the previous polytopal subdivision. For any $A \in \mc C$ of the form \eqref{eq:concentricpolytope}, we have $\ind L(A) < \ind L(S)$.

We can now proceed with proving the KMW theorem. Let $L \subset \bb Z^d$ be a $d$-dimensional lattice, and fix some nonzero element $x \in \bb Z^d/L$. Let $\mc C_x := \mc C_x^\circ \cup \mc C_x^\bullet$, where
\begin{itemize}
\item $\mc C_x^\circ$ is the set of all elements in $\mc C$ which do not have $x$ as a box point.
\item $\mc C_x^\bullet$ is the set of all elements in $\mc C$ whose standard form $(p,S,a)$ is such that $x$ is a box point of $S$, $\abs{p} = 1$, $\abs{S} = 1$, and the entry of $a$ is $d!$. 
\end{itemize}
We let $\mc C_x$ inherit a poset structure and poset map $\Cay : \mc C_x \to \mc P$ from $\mc C$. We can check that $(\mc C_x, \Cay)$ is a perfect $\mc P$-poset.

Let $\gamma_T^\circ$ be the restriction of $\gamma_T$ to $\mc C_x^\circ \cap \mc C_T$. It is easy to see that the image of $\gamma_T^\circ$ is contained in $\Subd(\mc C_x^\circ)$. Thus we have a well-defined subdivision $\gamma_T^\circ : \mc C_x^\circ \cap \mc C_T \to \Subd(\mc C_x)$ over $\mc P$.

Now, define a subdivision $\stell_x : \mc C_x^\bullet \to \Subd(\mc C_x)$ as follows. Let $A \in \mc C_x^\bullet$ with standard form $(p,S,a)$. We will abuse notation and identify $p$ and $S$ with their single entries. Let $c := c(S,x)$ and $x_0$ be the unique representative of $x +cv$ in $S$, where $v$ is the first vertex of $S$. Using \eqref{eq:concentricpolytope}, we define $\stell_x(A)$ to be the $\mc C$-subdivision whose maximal elements are
\[
\left( \left( p+rx_0, p+(r-1)x_0 \right),(F), \begin{pmatrix} (N-r)c \\ (N-r+1)c \end{pmatrix} \right)
\]
where $N = d!/c$, $1 \le r \le N$, and $F \in \mf F(S)$. As in \eqref{eq:concentricpolytope}, this is a $\mc C$-subdivision of $\Cay(A)$, and since $x$ is not a box point of $F$ by definition of $\mf F$, this is a $\mc C_x$-subdivision, as claimed.

We leave the following as an exercise:

\begin{prop}
The family of subdivisions $\{\gamma_T^\circ\}_T \cup \{\stell_x\}$ is locally confluent, terminating, and facially compatible (as $\mc C_x$-subdivisions).
\end{prop}

From this and the previous discussions we have the following.

\begin{thm} \label{thm:Gammax}
There is a canonical subdivision $\Gamma_x : \mc C_x \to \Subd(\mc C_x)$ over $\mc P$ such that for all $A \in \mc C_x$, we have the following:
\begin{enumerate}[label=(\alph*)]
\item If $A \in \mc C_x^\circ$, then $\Gamma_x(A) = \Gamma(A)$.
\item If $A \in \mc C_x^\bullet$ and $B \in \Max(\Gamma_x(A))$, where $(p,S,a)$ is the standard form of $B$, then $\abs{S} = 0$ and $\ind(L(B)) < \ind(L(A))$.
\end{enumerate}
\end{thm}

We can now give a proof of Theorem~\ref{thm:KMW}.

\begin{proof}[Proof of Theorem~\ref{thm:KMW}]
For any triangulation $X$, we will consider the collection of lattices
\[
\ms L(X) := \{ L(S) : S \in \Max X \}.
\]

Start with a $d$-dimensional triangulation $X$ in $\bb R^d$. Dilate the triangulation by $d!$ to form $d!X$. We can view $d!X$ as a $\mc C$-subdivision by replacing each simplex $d!S \in d!X$ by the element $((0), (S), (d!)) \in \mc C$. Call this $\mc C$-subdivision $Y$.

Now suppose there is a $d$-dimensional simplex $S \in X$ such that $\ind L(S) > 1$. Choose any nonzero $x \in \bb Z^d / L(S)$. Note that $Y \in \Subd(\mc C_x)$. Let $X' = \Cay \Gamma_x^\ast(Y)$, where $\Gamma_x$ is from Theorem~\ref{thm:Gammax} and the $\ast$ construction is from Proposition~\ref{prop:canonicalrefinement}. For every $T' \in X'$, we have that $T'$ is a simplex, and either
\begin{itemize}
\item $\ind L(T') = \ind L(T)$ for some $T \in X$ where $x$ is not a box point of $T$, or
\item $\ind L(T') < \ind L(T)$ for some $T \in X$ where $x$ is a box point of $T$.
\end{itemize}

Hence, comparing $\ms L(X)$ and $\ms L(X')$, we have that $\ms L(X')$ is obtained by replacing at least one lattice of $\ms L(X)$ with lattices of lower index, while keeping the other lattices the same. Thus, if we repeat the above process on $X'$, and so on, we will eventually obtain a triangulation $Z$ of $(d!)^N X$, for some $N$, such that $\ms L(Z)$ contains only the lattice $\bb Z^d$. Letting $X$ be any triangulation of an integral polytope $P$, we obtain the theorem.
\end{proof}

\subsection{The structure of box points} \label{sec:structboxpoints}

In this section we investigate the distribution of representatives of box points in dilated polysimplices.

Let $S = (S_1,\dots,S_n)$ be a tuple of ordered independent integral simplices and let $x$ be a box point of $S$. Let $c = c(S,x)$. By definition, the polysimplex $\sum_{j=1}^n c_j S_j$ contains a unique representative of $x + \sum_{j=1}^n c_j v_j$, where $v_j$ is the first vertex of $S_j$. Let $x_0$ be this representative. We call $x_0$ the \emph{focus} of $(S,x)$.

Let $\mf F = \mf F(S,x)$ be the set of all tuples $F = (F_1,\dots,F_n)$ such that $F \le S$, $\sum F_j$ is a facet of $\sum S_j$, and $x$ is not a box point of $F$. Thus, if $F \in \mf F$, the point $x_0$ is not contained in the facet $\sum_{j=1}^n c_j F_j$ of $\sum_{j=1}^n c_j S_j$. Moreover, the lattice distance from $x_0$ to $\sum_{j=1}^n c_j F_j$ is smaller than the lattice height of $\sum_{j=1}^n S_j$ with respect to the facet $\sum_{j=1}^n F_j$.

As in Section~\ref{sec:KMW}, the elements of the form
\[
\left((x_0,0),F, \begin{pmatrix} 0 & \cdots & 0 \\ c_1 & \cdots & c_n \end{pmatrix} \right)
\]
where $F \in \mf F$ are the maximal elements of a $\mc C$-subdivision of $\sum_{j=1}^n c_j S_j$. If $B$ is such an element, then $\ind L(B) < \ind (L(S_1) + \dots + L(S_n))$.

Now, consider the polysimplex $P = \sum_{j=1}^n a_j S_j$, where the $a_j$ are integers satisfying $a_j \ge c_j$ for all $j$. Since
\[
P = \sum_{j=1}^n c_j S_j + \sum_{j=1}^n (a_j - c_j) S_j\]
it follows that $P$ contains the polytope $P' := x_0 + \sum_{j=1}^n (a_j - c_j) S_j$. Moreover, if $F \in \mf F$, then the lattice distance between the facet $x_0 + \sum_{j=1}^n (a_j - c_j) F_j$ of $P'$ and the facet $\sum_{j=1}^n a_j F_j$ of $P$ is equal to lattice distance from $x_0$ to $\sum_{j=1}^n c_j F_j$.

Consider the elements of $\mc C$ of the form
\[
\left( (x_0,0),F, \begin{pmatrix} a_1-c_1 & \cdots & a_n-c_n \\ a_1 & \cdots & a_n \end{pmatrix} \right).
\]
where $F$ ranges over $\mf F$. These elements are the maximal elements of a $\mc C$-complex whose support is $\overline{P \setminus P'}$, the closure of $P \setminus P'$.
Note that if there is some $j \in \supp c$ such that $a_j = c_j$, then $P'$ has smaller dimension than $P$, and so in this case this complex is a $\mc C$-subdivision of $P$.
If $B \in \mc C$ is an element of the above form, then $\ind L(B) < \ind (L(S_1) + \dots + L(S_n))$.

If $a_j \ge 2c_j$ for all $j$, then repeating the above argument on $P'$, we have that the elements of $\mc C$ of the form
\[
\left( (2x_0,x_0),F, \begin{pmatrix} a_1-2c_1 & \cdots & a_n-2c_n \\ a_1-c_1 & \cdots & a_n-c_n \end{pmatrix} \right).
\]
are the maximal elements of a $\mc C$-complex whose support if $\overline{P' \setminus P''}$, where $P'' := 2x_0 + \sum_{j=1}^n (a_j-2c_j) F_j$. In general, if $a_j \ge Nc_j$ for some positive integer $N$, then the elements of $\mc C$ of the form
\[
\left( \left(rx_0,(r-1)x_0 \right),F,\begin{pmatrix} a_1-rc_1 & \cdots & a_n-rc_n \\ a_1-(r-1)c_1 & \cdots & a_n-(r-1)c_n \end{pmatrix} \right),
\]
where $r = 1$, \dots, $N$ and $F \in \mf F$, are the maximal elements of a $\mc C$-complex whose support is $\overline{P \setminus P^N}$, where $P^N := Nx_0 + \sum_{j=1}^n (a_j - Nc_j) F_j$. If in addition $a_j = Nc_j$ for some $j \in \supp c$, then this complex is a $\mc C$-subdivision of $P$. As before, if $B \in \mc C$ is a maximal element of this complex, then $\ind L(B) < \ind (L(S_1) + \dots + L(S_n))$.

\subsection{$\kappa_{T,x}$-subdivisions} \label{sec:kappaTx}

We will need one final geometric construction for the main proof. Like the previous constructions in this section, this construction uses box points to lower the indices of lattices, but in a way more analogous to $\gamma_T$. This is where our constructions begin to fundamentally diverge from the original proof of the KMW theorem.

Let $L \subset \bb Z^d$ be a $d$-dimensional lattice and let $x \in \bb Z^d / L$ be nonzero. Let $T$ be an ordered integral simplex with dimension at least one. Let $\mc K_{T,x}$ be the set of elements of $\mc C$ whose standard form $(p,S,a)$ satisfies the following:
\begin{itemize}
\item $T$ is an entry of $S$, say the $j_0$-th entry.
\item $x$ is a box point of $S$.
\item $a_{ij} \ge c(S,x)_{j}$ for all $i$ and all $j$.
\item $a_{ij_0} > c(S,x)_{j_0}$ for some $i$.
\end{itemize}

We now construct a subdivision $\kappa_{T,x} : \mc K_{T,x} \to \Subd(\mc C)$ over $\mc P$. Let $A \in \mc K_{T,x}$ with standard form $(p,S,a)$. Let $c = c(S,x)$. Let $j_0$ be the unique number such that $S_{j_0} = T$. Let $i_0$ be the smallest number such that $a_{i_0j_0} = \max_{i} a_{ij_0}$. Hence $a_{i_0j_0} > c_j$. To make the notation easier to read, we will assume from now on that $i_0 = j_0 = 1$; the below construction can be easily adjusted to allow for other values of $i_0$ and $j_0$.

Let $v$ be the first vertex of $T$, and let $f$ be the facet of $T$ opposite $v$.
Let $m = \abs{p}$ and $n = \abs{S}$. Let $F = (F_1,\dots,F_n)$ where $F_1 = f$ and $F_i = S_i$ for all $i > 1$.
Let $A' = (p',S,a')$ and $A'' = (p'',F,a'')$ be as in Section~\ref{sec:T}, using $i = i_0 = 1$ and $j = j_0 = 1$.

We consider two cases:

\textbf{Case 1:} $x$ is a box point of $F$.

In this case, we define $\kappa_{T,x}(A) = \gamma_T(A)$, as defined Section~\ref{sec:T}.

\textbf{Case 2:} $x$ is not a box point of $F$.

Let $\mf F = \mf F(S,x)$ be as in Section~\ref{sec:structboxpoints}, so $F \in \mf F$.
Let $x_0$ be the focus of $(S,x)$. Recall from Section~\ref{sec:structboxpoints} that the polytope $x_0 + \sum_{j=1}^n (a_{1j}-c_j)F_j$ is contained in $\sum_{j=1}^n a_{1j} S_j$. Moreover, the lattice distance between $x_0 + \sum_{j=1}^n (a_{1j}-c_j)F_j$ and $\sum_{j=1}^n a_{1j}F_j$ is less than the lattice height of $\sum_{j=1}^n S_j$ with respect to its facet $\sum_{j=1}^n F_j$. Hence, $x_0 + \sum_{j=1}^n (a_{1j}-c_j)F_j$ is contained in the polytope
\[
\conv \left( v + (a_{11}-1)F_1 + \sum_{j=2}^n a_{1j} F_j, \sum_{j=1}^n a_{1j}F_j \right).
\]
Consider the tuple $(p''', F, a''')$, where $p'''$ is obtained from $p$ by replacing the first entry with $p_1 + x_0$, and $a'''$ is obtained from $a$ by replacing the first row with $(a_{1j}-c_j)_{j=1}^n$. It follows from the above discussion that $\Cay(p''',F,a''') \subset \Cay(A'')$. The polytope $\Cay(p''',F,a''')$ is parallel to the facets $\Cay(p,F,a)$ and $\Cay(p',F,a')$ of $\Cay(A'')$.

We now define
\begin{equation} \label{eq:sharpflat}
\begin{aligned}
A^\sharp &= (p^\sharp,F,a^\sharp) \in \mc C \\
A^\flat &= (p^\flat,F,a^\flat) \in \mc C
\end{aligned}
\end{equation}
where
\begin{itemize}
\item $p^\sharp$ is the $(m+1)$-tuple obtained by inserting $p_1 + x_0$ directly before the $1$st entry of $p$.
\item $a^\sharp$ is the $(m+1) \times n$ matrix obtained by inserting the row $(a_{1j}-c_j)_{j=1}^n$ directly above the $1$st row of $a$.
\item $p^\flat$ is the $(m+1)$-tuple obtained by inserting $p_1 + x_0$ directly before the $1$st entry of $p'$.
\item $a^\flat$ is the $(m+1) \times n$ matrix obtained by inserting the row $(a_{1j}-c_j)_{j=1}^n$ directly above the $1$st row of $a'$.
\end{itemize}
We have that
\begin{align*}
\Cay(A^\sharp) &= \conv( \Cay(p''',F,a'''), \Cay(p,F,a)) \\
\Cay(A^\flat) &= \conv( \Cay(p''',F,a'''), \Cay(p',F,a')).
\end{align*}

Now, let $\mf G$ be the set of all tuples $G = (G_1,\dots,G_n)$ such that $G \le F$, $\sum_{j=1}^n G_j$ is a facet of $\sum_{j=1}^n F_j$, and $G \le F'$ for some $F' \in \mf F$ with $F' \neq F$. Let 
\begin{equation} \label{eq:AG}
A^G = (p^G, G, a^G) \in \mc C
\end{equation}
where
\begin{itemize}
\item $p^G$ is the $(m+2)$-tuple obtained by inserting $p_1+x_0$ directly before the 1st entry of $p''$.
\item $a^G$ is the $(m+2)\times n$ matrix obtained by inserting the row $(a_{1j}-c_j)_{j=1}^n$ directly above the $1$st row of $a''$.
\end{itemize}
Then the elements $A^\sharp$, $A^\flat$, and $A^G$ over all $G \in \mf G$ are the maximal elements of a $\mc C$-subdivision of $\Cay(A'')$. Thus, the elements $A'$, $A^\sharp$, $A^\flat$, and $A^G$ over all $G \in \mf G$ are the maximal elements of a $\mc C$-subdivision of $\Cay(A)$. We define $\kappa_{T,x}(A)$ to be this subdivision.

We prove the following two properties of $\kappa_{T,x}$.

\begin{prop} \label{prop:kappaTx}
Let $A \in \mc K_{T,x}$ and $B \in \Max \kappa_{T,x}(A)$. We have the following:
\begin{itemize}
\item If $x$ is a box point of $B$, then $L(A) = L(B)$.
\item If $x$ is not a box point of $B$, then $\ind(L(B)) < \ind(L(A))$.
\end{itemize}
\end{prop}

\begin{proof}
If we are in Case 1, then $x$ is a box point of $B$ and $L(A) = L(B)$ by Proposition~\ref{prop:LgammaT}. Assume we are in Case 2. If $B = A'$, then $x$ is a box point of $B$ and $L(A) = L(B)$ by Proposition~\ref{prop:LgammaT}. So assume $B = A^\sharp$, $A^\flat$, or $A^G$ for some $G \in \mf G$. Thus $x$ is not a box point of $B$.

Let $F$, $\mf F$ and $x_0$ be as above. For each $E \in \mf F$, let $h_E$ be the lattice height of $\sum S_j$ with respect to its facet $\sum E_j$. Let $h_E'$ be the lattice distance between $x_0$ and $\sum E_j$. Hence $0 < h_E' < h_E$ for all $E \in \mf F$.

If $B = A^\sharp$, then we have
\[
\ind(L(B)) = h_F' \ind(L(F)) < h_F \ind(L(F)) = \ind(L(A)).
\]
If $B = A^\flat$, then we have
\[
\ind(L(B)) = (h_F-h_F') \ind(L(F)) < h_F \ind(L(F)) = \ind(L(A)).
\]
Finally, suppose $B = A^G$ for some $G \in \mf G$. Let $E \in \mf F$ such that $G \le E$ and $E \neq F$. Then
\begin{multline*}
\ind(L(B)) = h_E' \ind(L(p'',G,a'')) = h_E' \ind(L(E))  \\ < h_E \ind(L(E)) = \ind(L(A)).
\end{multline*}
In all cases $\ind(L(B)) < \ind(L(A))$, completing the proof.
\end{proof}

\begin{prop} \label{prop:facekappaTx}
Let $A \in \mc K_{T,x}$ and suppose $\mc B \in \mc C$ such that $B \le_{\mc C} A$. Then one of the following hold:
\begin{itemize}
\item $\kappa_{T,x}(A|B) = \kappa_{T',x}(B)$ for some $T'$
\item $\kappa_{T,x}(A|B) = \gamma_{T'}(B)$ for some $T'$
\item $\kappa_{T,x}(A|B) = \triv_{\mc D}(B)$.
\end{itemize}
\end{prop}

\begin{proof}
Let the standard form of $A$ be $(p,S,a)$. Let $B = (p_I,S',a_{I \times \bullet})$ with $S' \le S$ and $I$ a nonempty subset of $[\abs{p}]$. Assume that $T = S_j$, and let $i$ be the smallest number such that $a_{ij} = \max_{i'} a_{i'j}$. If $i \in I$ and $S_j'$ contains the first vertex of $T$ and another vertex, then
\[
\kappa_{T,x}(A|B) = \begin{dcases*}
\kappa_{S_j',x}(B) & if $x$ is a box point of $S'$ \\
\gamma_{S_j'}(B) & if $x$ is not a box point of $S'$
\end{dcases*}
\]
Otherwise, we have $\gamma_T(A|B) = \triv_{\mc C}(B)$, as desired.
\end{proof}

\section{Main proof} \label{sec:main}

Let $d$ be a positive integer and let $c$ be any integer greater than or equal to $d!+d$. In this section we prove the following.

\begin{thm} \label{thm:main2}
For any $d$-dimensional integral polytope $P$ in $\bb R^d$, there exists a positive integer $N$ such that for all nonnegative integers $r$ and $s$, $(rc^N + s(d!)^N)P$ has a unimodular triangulation.
\end{thm}

Taking $c$ to be relatively prime to $d!$, this theorem will imply Theorem~\ref{thm:main}, as desired.

The overall strategy of the proof is similar to the proof of the KMW theorem in Section~\ref{sec:KMW}. We will fix a box point $x$, and then construct two types of subdivisions: one that preserves lattices that do not have $x$ as a box point, and one that replaces lattices which do have $x$ as a box point with lattices of smaller index. There are two main differences between the current proof and that of the KMW theorem. The first is that we work with intermediate dilations by $c$, rather than just $d!$. To deal with this we will use the subdivisions $\kappa_{T,x}$ defined earlier. The second difference is that we work over $2\mc P$. Roughly, this raises the following problem: We might produce an element $(P,Q) \in 2 \mc P$ where $P$ and $Q$ are independent and $\ind(P+Q) > \ind(P)\ind(Q)$. Such an element can never have a mixed subdivision into elements of index 1, even after dilation. Thus, we must avoid such elements. The way this is done is to work within a carefully defined $2\mc P$-poset $\mc D$ which always avoids such elements, and construct subdivisions which always remain in $\mc D$. These restrictions are the reason why this proof is much more complicated than the proof of the KMW theorem or the arguments in \cite{ALT18}.

\subsection{The poset $\mc D$} \label{sec:posetD}

\subsubsection{The poset $\mc M$}

We first define a poset $\mc M_0$ as follows. Its elements are all elements
\[
(p,S,a) \times (q,S,b)  \times k \in \mc C_0 \times \mc C_0 \times \bb Z,
\]
satisfying the following properties.
\begin{enumerate}[label=(\Roman*)]
\item $\abs{p} = \abs{q}$.
\item $0 \le k \le \abs{S}$.
\item $p_i-q_i$ is constant over all $i$
\item \label{samedifference} For each $j \le k$, we have $a_{ij} = b_{ij}$ for all $i$.
\item For each $j > k$, $a_{ij} = 1$ for all $i$, and either $b_{ij} = 0$ for all $i$ or $b_{ij} = 1$ for all $i$. 
\end{enumerate}

We define an equivalence relation on $\mc M_0$ as follows. If $j$ is such that $j \le k$ and either $S_j$ is a point or $a_{ij} = 0$ for all $i$, we set
\begin{multline*}
(p,S,a) \times (q,S,b) \times k \sim \\
((p_i+a_{ij}S_j)_{i \in [\abs{p}]}, S_{\setminus j}, a_{\bullet \times \setminus j}) \times ((q_i+b_{ij}S_j)_{i \in [\abs{q}]}, S_{\setminus j}, b_{\bullet \times \setminus j}) \times (k-1).
\end{multline*}
Also, if $j$ is such that $j > k$ and $S_j$ is a point, then we set
\begin{multline*}
(p,S,a) \times (q,S,b) \times k \sim \\
((p_i+a_{ij}S_j)_{i \in [\abs{p}]}, S_{\setminus j}, a_{\bullet \times \setminus j}) \times ((q_i+b_{ij}S_j)_{i \in [\abs{q}]}, S_{\setminus j}, b_{\bullet \times \setminus j}) \times k.
\end{multline*}
(Note that if $j > k$, then by definition $a_{ij} \neq 0$ for any $i$.) We define $\mc M := \mc M_0 / \sim$. We define the \emph{standard form} of an element $A \in \mc M$ to be the unique representative $(p,S,a) \times (q,S,b) \times k \in \mc M_0$ of $A$ such that $(p,S,a)$ is in standard form.

For $A$, $B \in \mc M$ we let $B \le_{\mc M} A$ if the standard form of $A$ is $(p,S,a) \times (q,S,b) \times k$ and $B$ has a representative of the form
\[
(p_I,S',a_{I \times \bullet}) \times (q_I,S',b_{I \times \bullet}) \times k
\]
where $I$ is a nonempty subset of $[\abs{p}]$ and $S' \le S$. Then $\le_{\mc M}$ is a partial order on $\mc M$.

Let $\Cay : \mc M \to 2\mc P$ be the map $\Cay(U \times V \times k) = (\Cay(U),\Cay(V))$. This is a well-defined poset map. Hence $(\mc M,\Cay)$ is a $2\mc P$-poset. Furthermore, $(\mc M,\Sum \circ \Cay)$ is a perfect $\mc P$-poset.

\subsubsection{The poset $\mc N$}

We similarly define a poset $\mc N$ as follows. Its elements are all elements
\[
(p,S,a) \times (q,S,b)  \times k \in \mc C_0 \times \mc C_0 \times \bb Z,
\]
satisfying the following properties.
\begin{enumerate}[label=(\Roman*), resume]
\item $\abs{q} = 1$.
\item $0 \le k \le \abs{S}$.
\item $a_{ij} \ge b_{1j}$ for all $i$, $j$.
\item For each $j > k$, $a_{ij} = 1$ for all $i$, and $b_{1j} = 0$ or 1.
\end{enumerate}
We define equivalence relation $\sim$ on $\mc N_0$ analogously to the previous section and define $\mc N = \mc N_0 / \sim$. As before, we define the \emph{standard form} of an element $A \in \mc N$ to be the unique representative $(p,S,a) \times (q,S,b) \times k \in \mc N_0$ of $A$ such that $(p,S,a)$ is in standard form.

For $A$, $B \in \mc N$ we let $B \le_{\mc N} A$ if the standard form of $A$ is $(p,S,a) \times (q,S,b) \times k$ and $B$ has a representative of the form
\[
(p_I,S',a_{I \times \bullet}) \times (q,S',b) \times k
\]
where $I$ is a nonempty subset of $[\abs{p}]$ and $S' \le S$. Then $\le_{\mc N}$ is a partial order on $\mc N$.

As before, let $\Cay : \mc N \to 2\mc P$ be the map $\Cay(U \times V \times k) = (\Cay(U),\Cay(V))$. This is a well-defined poset map. Hence $(\mc N,\Cay)$ is a $2\mc P$-poset. Furthermore, $(\mc N,\Sum \circ \Cay)$ is a perfect $\mc P$-poset.

\subsubsection{The poset $\mc D$}

The set $\mc M \cap \mc N$ is a poset ideal of both $\mc M$ and $\mc N$; specifically, it is the set of elements $(p,S,a) \times (q,S,b) \times k$ in $\mc M$ or $\mc N$ such that $\abs{p} = 1$ and \ref{samedifference} holds. Moreover, if $A$, $B \in \mc M \cap \mc N$, then $A \le_{\mc M} B$ if and only if $A \le_{\mc N} B$. Hence, we can define a poset $\mc D$ on the set $\mc M \cup \mc N$ where $A \le_{\mc D} B$ if and only if $A \le_{\mc M} B$ or $A \le_{\mc N} B$.

As above, we define a map $\Cay : \mc D \to 2 \mc P$ by $\Cay(U \times V \times k) = (\Cay(U),\Cay(V))$. Then $(\mc D,\Cay)$ is a $2\mc P$-poset, and $(\mc D,\Sum \circ \Cay)$ is a perfect $\mc P$-poset.

For $A \in \mc D$ with standard form $(p,S,a) \times (q,S,b) \times k$, we define $L(A) := L(p,S,a)$.
We say that $x$ is a box point of $A$ if $x$ is a box point of $S_{[k]}$.

\subsection{$\mu_T$, $\nu_T$, and $\epsilon$ subdivisions}

In this section we define three subdivisions on subposets of $\mc D$. These subdivisions play the role of $\gamma_T$ in that they preserve lattices.

\subsubsection{$\mu_T$ subdivisions}

Let $T$ be an ordered integral simplex of dimension at least 1. Let $\mc M_T$ be the set of elements of $\mc M$ whose standard form $(p,S,a) \times (q,S,b) \times k$ has the property that $T$ is an entry of $S_{[k]}$.

We construct a subdivision as follows $\mu_T : \mc M_T \to \Subd(\mc D)$. Let $A \in \mc M_T$ with standard form $U \times V \times k = (p,S,a) \times (q,S,b) \times k$. Assume $T$ is the $j$-th entry of $S$, where $j \le k$. We have that $U \in \mc C_T$, where $\mc C_T$ is as in Section~\ref{sec:T}. Also, by \ref{samedifference}, we have that the $j$-th column of $b$ is not all 0, so $V \in \mc C_T$.

Following the construction of $\gamma_T(U)$ and $\gamma_T(V)$, let $U'$, $U''$, $V'$, and $V''$ be as in \eqref{eq:A'A''}. Then $U' \times V' \times k$, $U'' \times V'' \times k \in \mc M$. We define
\[
\mu_T(A) = \triv_{\mc D}(U' \times V' \times k) \cup \triv_{\mc D}(U'' \times V'' \times k).
\]
Then $\mu_T(A)$ is a $\mc D$-subdivision with 2-support $\Cay(U) \times \Cay(V) = \Cay(A)$. Hence, $\mu_T : \mc M_T \to \Subd(\mc D)$ is a subdivision over $2 \mc P$.

From Proposition~\ref{prop:LgammaT}, we have the following.

\begin{prop} \label{prop:LmuT}
Let $A \in \mc M_T$ and $B \in \Max \mu_T(A)$. Then $L(A) = L(B)$.
\end{prop}

From the proof of Proposition~\ref{prop:facegammaT}, we have the following.

\begin{prop} \label{prop:facemuT}
Let $A \in \mc M_T$ and suppose $\mc B \in \mc D$ such that $B \le_{\mc D} A$. Then either $\mu_T(A|B) = \mu_{T'}(B)$ for some $T'$ or $\mu_T(A|B) = \triv_{\mc D}(B)$.
\end{prop}

\subsubsection{$\nu_T$ subdivisions}

Let $T$ be an ordered integral simplex of dimension at least 1. Let $\mc N_T$ be the set of elements of $\mc N$ whose standard form $(p,S,a) \times (q,S,b) \times k$ has the property that
\begin{enumerate}[label=(\alph*)]
\item $T$ is an entry of $S_{[k]}$, say the $j$-th entry.
\item There is some $i$ such that $a_{ij} > b_{1j}$.
\end{enumerate}

We construct a subdivision as follows $\nu_T : \mc N_T \to \Subd(\mc D)$. Let $A \in \mc N_T$ with standard form $U \times V \times k = (p,S,a) \times (q,S,b) \times k$. Following the construction of $\gamma_T(U)$, let $U'$, $U''$ be as in \eqref{eq:A'A''}. Let $V_F := (q,F,b)$, where $F$ is as in \eqref{eq:A'A''}. Then $U' \times V \times k$, $U'' \times V_F \times k \in \mc N$. We define
\[
\nu_T = \triv_{\mc D}(U' \times V \times k) \cup \triv_{\mc D}(U'' \times V_F \times k)
\]

Then $\nu_T(A)$ is a $\mc D$-subdivision with 2-support $\Cay(U) \times \Cay(V) = \Cay(A)$. Hence, $\nu_T : \mc N_T \to \Subd(\mc D)$ is a subdivision over $2 \mc P$.

From Proposition~\ref{prop:LgammaT}, we have the following.

\begin{prop} \label{prop:LnuT}
Let $A \in \mc N_T$ and $B \in \Max \nu_T(A)$. Then $L(A) = L(B)$.
\end{prop}

From the proof of Proposition~\ref{prop:facegammaT}, we have the following.

\begin{prop} \label{prop:facenuT}
Let $A \in \mc N_T$ and suppose $\mc B \in \mc D$ such that $B \le_{\mc D} A$. Then either $\nu_T(A|B) = \nu_{T'}(B)$ for some $T'$ or $\nu_T(A|B) = \triv_{\mc D}(B)$.
\end{prop}

\subsubsection{$\epsilon$ subdivisions} \label{sec:epsilon}

Finally, let $\mc E$ be the set of elements $A \in \mc D$ whose standard form $(p,S,a) \times (q,S,b) \times k$ satisfies the following: $\abs{p} > 1$, and one of the following hold:
\begin{enumerate}[label=(\roman*)]
\item $A \in \mc M$, and $k = 0$.
\item $A \in \mc N$, and for all $j \le k$ we have $a_{ij} = b_{ij}$ for all $i$.
\end{enumerate}
Note that in either case, all the rows of $a$ are equal to each other and all the rows of $b$ are equal to each other.

We construct a subdivision $\epsilon : \mc E \to \mc D$ as follows. Let $A \in \mc E$ with standard form $(p,S,a) \times (q,S,b) \times k$. Note that the entries of $p$ are affinely independent, because by definition of $\mc C_0$ the polytopes $p_i + \sum_{j=1}^n a_{ij}S_j$ are in Cayley position and $\sum_{j=1}^n a_{ij}S_j$ is constant over all $i$. Let $T$ be the ordered simplex with vertices given by the entries of $p$, in that order.

Given a tuple $c$ and an object $d$, let $c \cup d$ denote the tuple obtained by concatenating $d$ to the end of $c$. Let $a_1$ denote the first row of $a$ and define $b_1$ similarly. If $A \in \mc M$, we define
\[
A^T := ( (0), S \cup T, a_1 \cup 1 ) \times ((q_1-p_1), S \cup T, b_1 \cup 1) \times k \in \mc D.
\]
If $A \in \mc N$, we define
\begin{equation}
A^T := ( (0), S \cup T, a_1 \cup 1 ) \times (q, S \cup T, b \cup 0) \times k \in \mc D.
\end{equation}
Either way, we have $\Cay(A^T) = \Cay(A)$. We define $\epsilon(A) = \triv_{\mc D}(A^T)$. We have that $\epsilon : \mc E \to \mc D$ is a subdivision over $2 \mc P$. 

We have the following two properties of $\epsilon$.

\begin{prop} \label{prop:Lepsilon}
Let $A \in \mc E$ and $B \in \Max \epsilon(A)$. Then $L(A) = L(B)$.
\end{prop}

\begin{proof}
This follows from the observation that $T$ is a translation of $S_0(p,S,a)$.
\end{proof}

From the proof of Proposition~\ref{prop:facegammaT}, we have the following.

\begin{prop} \label{prop:faceepsilon}
Let $A \in \mc E$ and suppose $\mc B \in \mc D$ such that $B \le_{\mc D} A$. Then either $\epsilon(A|B) = \epsilon(B)$ or $\epsilon(A|B) = \triv_{\mc D}(B)$.
\end{prop}

\begin{proof}
Let the standard form of $A$ be $(p,S,a) \times (q,S,b) \times k$ and let $B = U \times V \times W$, where $U = (p_I,S',a_{I \times \bullet})$ for some nonempty $I \subseteq [\abs{p}]$ and $S' \le S$. If $\abs{I} > 1$, then $B \in \mc E$ and $\epsilon(A|B) = \epsilon(B)$. Otherwise, $\epsilon(A|B) = \triv_{\mc D}(B)$.
\end{proof}

\subsection{$\tau_x$, $\sigma_x$, and $\rho_x$ subdivisions}

The final subdivisions we will construct are analogues of the $\stell_x$ and $\kappa_{T,x}$ subdivisions from earlier. These will allow us to lower indices of lattices. The order they are presented here is ``backwards'', in the sense that in practice, one would apply $\rho_x$ subdivisions first, then $\sigma_x$ subdivisions, then $\tau_x$.

Throughout this section, we fix $L$ a $d$-dimensional lattice in $\bb R^d$ and some nonzero $x \in \bb Z^d / L$.

\subsubsection{$\tau_x$ subdivisions}

Let $\mc T_x$ be the set of elements of $\mc D$ whose standard form $(p,S,a) \times (q,S,b) \times k$ satisfies the following:
\begin{enumerate}[label=(\alph*)]
\item $x$ is a box point of $S_{[k]}$.
\item $\abs{p} = 1$.
\item For all $j \le k$, $a_{1j} = b_{1j} = d!$.
\end{enumerate}

We construct a subdivision $\tau_x : \mc T_x \to \Subd(\mc D)$ as follows. Let $A \in \mc T_x$ with standard form $(p,S,a) \times (q,S,b) \times k$. Let $n = \abs{S}$. Let $c$ be the $n$-tuple of integers where
\[
c_j = \begin{dcases*}
c(S_{[k]},x)_j & for $1 \le j \le k$ \\
0 & for $k+1 \le j \le n$
\end{dcases*}
\]

Now, let $x_0$ be the focus of $(S,x)$, and let $\mf F = \mf F(S,x)$.
Let $N = d!/\max_j c_j$. For all $F \in \mf F$ and $r = 1$, \dots, $N$, define
\begin{align*}
U_{F,r} &:= \left( \left(p_1+rx_0,p_1+(r-1)x_0 \right),F,\right. \\
& \qquad\qquad\qquad\qquad\qquad\qquad\qquad\quad \left.
\begin{pmatrix} a_{11}-rc_1 & \cdots & a_{1n}-rc_n \\ a_{11}-(r-1)c_1 & \cdots & a_{1n}-(r-1)c_n \end{pmatrix} \right) \\
V_{F,r} &:= \left( \left(q_1+rx_0,q_1+(r-1)x_0 \right),F, \right. \\
& \qquad\qquad\qquad\qquad\qquad\qquad\qquad\quad \left. \begin{pmatrix} b_{11}-rc_1 & \cdots & b_{1n}-rc_n \\ b_{11}-(r-1)c_1 & \cdots & b_{1n}-(r-1)c_n \end{pmatrix} \right).
\end{align*}

We have that $U_{F,r} \times V_{F,r} \times k \in \mc M$.
From Section~\ref{sec:structboxpoints}, the set of all $U_{F,r} \times V_{F,r} \times k$ over $F \in \mf F$ and $r \in [N+1]$ is the set of maximal elements of a $\mc D$-subdivision with 2-support $\Cay(p,S,a) \times \Cay(q,S,b) = \Cay(A)$. We define $\tau_x(A)$ to be this $\mc D$-subdivision.

Hence, we have constructed a subdivision $\tau_x :\mc T_x \to \Subd(\mc D)$ over $2 \mc P$. From Section~\ref{sec:structboxpoints}, we have the following:

\begin{prop} \label{prop:Ltaux}
Let $A \in \mc T_x$ and $B \in \Max \tau_x(A)$. Then $L(B) < L(A)$.
\end{prop}

We also have the following.

\begin{prop} \label{prop:facetaux}
Let $A \in \mc T_x$ and suppose $\mc B \in \mc D$ such that $B \le_{\mc D} A$. Then either $\tau_x(A|B) = \tau_x(B)$ or $\tau_x(A|B) = \triv_{\mc D}(B)$.
\end{prop}

\begin{proof}
If $x$ is a box point of $B$, then $\tau_x(A|B) = \tau_x(B)$. Otherwise, $\tau_x(A|B) = \triv_{\mc D}(B)$.
\end{proof}

\subsubsection{$\sigma_x$ subdivisions}

Let $\mc S_x$ be the set of elements of $\mc D$ whose standard form $(p,S,a) \times (q,S,b) \times k$ satisfies the following:
\begin{enumerate}[label=(\alph*)]
\item $x$ is a box point of $S_{[k]}$.
\item $(p,S,a) \times (q,S,b) \times k \in \mc N$.
\item For all $j \le k$, $b_{1j} = 0$ or $d!$.
\item For all $i$, either $a_{ij} = b_{1j}$ for all $j \le k$ or $a_{ij} = b_{1j} + c(S_{[k]},x)_j$ for all $j \le k$.
\item There is some $i$ such that $a_{ij} = b_{1j} + c(S_{[k]},x)_j$ for all $j \le k$.
\end{enumerate}

We construct a subdivision $\sigma_x : \mc S_x \to \Subd(\mc D)$ as follows. Let $A \in \mc S_x$ with standard form $U \times V \times k = (p,S,a) \times (q,S,b) \times k$. Let $m = \abs{p}$ and $n = \abs{S}$. Let $c$ be the $n$-tuple of integers where $c_j = c(S_{[k]},x)_j$ for $j \le k$ and $c_j = 0$ for $j > k$, as before.
Let $x_0$ be the focus of $(S,x)$, and let $\mf F = \mf F(S,x)$.

Let $i$ be the smallest number such that $a_{ij} = b_{1j} + c_j$ for all $j \le k$. For all $F \in \mf F$, we define
\[
A^{\sharp,F} := U^{\sharp,F} \times V_F \times k \in \mc N
\]
where $V_F = (q,F,b)$, and $U^{\sharp,F} = (p^{\sharp},F,a^{\sharp})$ is defined analogously to $A^\sharp$ in \eqref{eq:sharpflat}; in other words,
\begin{itemize}
\item $p^{\sharp}$ is the $(m+1)$-tuple obtained by inserting $p_1 + x_0$ directly before the $i$th entry of $p$.
\item $a^{\sharp}$ is the $(m+1) \times n$ matrix obtained by inserting the row $(a_{ij}-c_j)_{j=1}^n$ directly above the $i$th row of $a$.
\end{itemize}
In addition, we define
\[
A^\star := (p''', S, a''') \times V \times k \in \mc N
\]
where $p'''$ and $a'''$ are defined right before \eqref{eq:sharpflat}; that is,
\begin{itemize}
\item $p'''$ is the $m$-tuple obtained by replacing the $i$th entry of $p$ with $p_1 + x_0$.
\item $a'''$ is the $m \times n$ matrix obtained by replacing the $i$th row with $(a_{ij}-c_j)_{j=1}^n$.
\end{itemize}
From the discussion in Section~\ref{sec:structboxpoints}, $\{A^\star\} \cup \{A^{\sharp,F} : F \in \mf F\}$ is the set of maximal elements of a $\mc D$-subdivision with 2-support $\Cay(U) \times \Cay(V) = \Cay(A)$. We define $\sigma_x(A)$ to be this $\mc D$-subdivision.

Hence, we have constructed a subdivision $\sigma_x :\mc S_x \to \Subd(\mc D)$ over $2 \mc P$. Using similar arguments as the previous section, we have the following:

\begin{prop} \label{prop:Lsigmax}
Let $A \in \mc S_x$ and $B \in \Max \sigma_x(A)$. Then $L(B) < L(A)$.
\end{prop}

\begin{prop} \label{prop:facesigmax}
Let $A \in \mc S_x$ and suppose $\mc B \in \mc D$ such that $B \le_{\mc D} A$. Then either $\sigma_x(A|B) = \sigma_x(B)$ or $\sigma_x(A|B) = \triv_{\mc D}(B)$.
\end{prop}

\subsubsection{$\rho_x$ subdivisions.}

Let $\mc R_x$ be the set of elements of $\mc D$ whose standard from $(p,S,a) \times (q,S,b) \times k$ satisfies the following.
\begin{enumerate}[label=(\alph*)]
\item $x$ is a box point of $S$.
\item $(p,S,a) \times (q,S,b) \times k \in \mc N$.
\item For all $j \le k$, either $b_{1j} = 0$ or $b_{1j} = d!$.
\item For all $i$ and all $j \le k$, we have $a_{ij} \ge b_{1j} + c(S,x)_j$.
\item There exists some $i$ and some $j \le k$ satisfying $a_{ij} > b_{1j} + c(S,x)_j$.
\end{enumerate}

We construct a subdivision $\rho_x : \mc R_x \to \Subd(\mc D)$ as follows. Let $A \in \mc R_x$ with standard form $U \times V \times k = (p,S,a) \times (q,S,b) \times k$. Let $j \le k$ be the smallest number such that there exists $i$ satisfying $a_{ij} > b_{1j} + c(S,x)_j$. Let $T = S_j$. Let $f$ be the facet of $T$ opposite the first vertex of $T$, and let $F$ be the tuple obtained from $S$ by replacing $T$ with $f$. We consider two cases, in parallel to Section~\ref{sec:kappaTx}.

\textbf{Case 1:} $x$ is a box point of $F$.

In this case, we define $\rho_x(A) = \nu_T(A)$.

\textbf{Case 2:} $x$ is not a box point of $F$.

Following the construction of $\kappa_{T,x}$, define $U'$ as in \eqref{eq:A'A''}, and define $U^\sharp$, $U^\flat$ as in \eqref{eq:sharpflat}. In addition, define $\mf G$ as in Section~\ref{sec:kappaTx}, and for each $G \in \mf G$ define $U^G$ as in \eqref{eq:AG}. For any $R \le S$, let $V_R := (q,R,b)$. Then the set of elements
\[
\{ U' \times V \times k, U^\sharp \times V_F \times k, U^\flat \times V_F \times k \} \cup \{ U^G \times V_G \times k : G \in \mf G \} \subset \mc N
\]
is the set of maximal elements of a $\mc D$-subdivision with 2-support $\Cay(A)$. We define $\rho_x(A)$ to be this $\mc D$-subdivision. Hence we have defined a subdivision $\rho_x : \mc R_x \to \Subd(\mc D)$ over $2 \mc P$.

From Proposition~\ref{prop:kappaTx}, we have the following.

\begin{prop} \label{prop:Lrhox}
Let $A \in \mc R_x$ and $B \in \Max \rho_x(A)$. We have the following.
\begin{itemize}
\item If $x$ is a box point of $B$, then $L(A) = L(B)$.
\item If $x$ is not a box point of $B$, then $\ind(L(B)) < \ind(L(A))$.
\end{itemize}
\end{prop}

Finally, we have the following, with an analogous proof to Proposition~\ref{prop:facekappaTx}.

\begin{prop} \label{prop:facerhox}
Let $A \in \mc R_x$ and suppose $\mc B \in \mc D$ such that $B \le_{\mc D} A$. Then one of the following hold:
\begin{itemize}
\item $\rho_x(A|B) = \rho_x(B)$
\item $\rho_x(A|B) = \nu_T(B)$ for some $T$
\item $\rho_x(A|B) = \triv_{\mc D}(B)$.
\end{itemize}
\end{prop}

%\subsubsection{$\epsilon_x$ subdivisions}
%
%Finally, let $L$ be a $d$-dimensional lattice in $\bb R^d$ and let $x \in \bb Z^d / L$ be nonzero as before. Let $\mc E_x$ be the subposet of $\mc D$ induced on the elements whose standard form $(p,S,a) \times (q,S,b)$ satisfies the following:
%\begin{enumerate}[label=(\alph*)]
%\item $x$ is a box point of $S$.
%\item $(p,S,a) \times (q,S,b) \in \mc N$ and $\abs{p} > 1$.
%\item For all $j$, either $b_{1j} = 0$ or $b_{1j} = d!$.
%\item For all $j$, we have $a_{ij} = \max(b_{1j} + c(S,x)_j, 1)$ for all $i$.
%\end{enumerate}
%Note that if $(p,S,a) \times (q,S,b) \in \mc E_x$, then all rows of $a$ are equal to each other.
%
%We construct a subdivision $\epsilon_x : \mc E \to \mc D$ as follows. Let $A \in \mc E_x$. Let $A^\Delta$ be as in \eqref{eq:ADelta}. We define $\epsilon_x(A) = \triv_{\mc D}(A^\Delta)$, exactly as in Section~\ref{sec:epsilon}. We have that $\epsilon_x : \mc E \to \mc D$ is a subdivision over $2 \mc P$. As in Section~\ref{sec:epsilon}, we have the following.
%
%\begin{prop} \label{prop:Lepsilonx}
%Let $A \in \mc E_x$ and $B \in \Max \epsilon_x(A)$. Then $L(A) = L(B)$.
%\end{prop}
%
%\begin{prop} \label{prop:faceepsilonx}
%Let $A \in \mc E_x$ and suppose $\mc B \in \mc D$ such that $B \le_{\mc D} A$. Then either $\epsilon_x(A|B) = \epsilon_x(B)$ or $\epsilon_x(A|B) = \triv_{\mc D}(B)$.
%\end{prop}

\subsubsection{The set $\mc D^\bullet_x$}

Let $\mc D^\bullet_x$ be the set of elements of $\mc D$ whose standard form $(p,S,a) \times (q,S,b) \times k$ satisfies the following:
\begin{enumerate}[label=(\alph*)]
\item $x$ is a box point of $S_{[k]}$.
\item $(p,S,a) \times (q,S,b) \times k \in \mc N$.
\item For all $j \le k$, $b_{1j} = 0$ or $d!$.
\item For all $i$, either $a_{ij} = b_{1j}$ for all $j \le k$ or $a_{ij} \ge b_{1j} + c(S_{[k]},x)_j$ for all $j \le k$.
\end{enumerate}

Observe that the sets $\mc R_x$, $\mc S_x$, $\mc T_x$, and $\mc E \cap \mc D^\bullet_x$ are pairwise disjoint and partition $\mc D^\bullet_x$.

\subsection{Proof of Theorem~\ref{thm:main2}}

With our constructions completed, we are now ready to prove Theorem~\ref{thm:main2}. The proof mirrors our proof of the KMW theorem from Section~\ref{sec:KMW}.

Let $L$ be a $d$-dimensional lattice in $\bb R^d$, and fix some nonzero element $x \in \bb Z^d / L$. Define
\[
\mc D_x := \mc D_x^\circ \cup \mc D_x^\bullet, 
\]
where $\mc D_x^\circ$ is the set of all elements of $\mc D$ which do not have $x$ as a box point. We let $\mc D_x$ inherit a poset structure and poset map $\Cay : \mc D_x \to 2\mc P$ from $\mc D$. We can check that $(\mc D_x,\Cay)$ is a perfect $\mc P$-poset.

Let $\mu_T^\circ$ and $\nu_T^\circ$ be the restrictions of $\mu_T$ and $\nu_T$, respectively, to
\begin{align*}
\mc M_T^\circ &:= \mc M_T \cap \mc D_x^\circ \\
\mc N_T^\circ &:= \mc N_T \cap \mc D_x^\circ
\end{align*}
respectively. Let $\epsilon_x$ be the restriction of $\epsilon$ to $\mc E_x := \mc E \cap \mc D_x$.
We can check that for each of the subdivisions $\mu_T^\circ$, $\nu_T^\circ$, $\epsilon_x$, $\rho_x$, $\sigma_x$, and $\tau_x$, the output is always a $\mc D_x$-subdivision. Thus we have well defined subdivisions 
\begin{align*}
\mu_T^\circ &: \mc M_T^\circ \to \Subd(\mc D_x) \\
\nu_T^\circ &: \mc N_T^\circ \to \Subd(\mc D_x) \\
\epsilon_x &: \mc E_x \to \Subd(\mc D_x) \\
\rho_x &: \mc R_x \to \Subd(\mc D_x) \\
\sigma_x &: \mc S_x \to \Subd(\mc D_x) \\
\tau_x &: \mc T_x \to \Subd(\mc D_x)
\end{align*}
over $2\mc P$. We now have the following.

\begin{prop} \label{prop:S}
The family of subdivisions
\[
\ms S_x := \{\mu_T^\circ\}_T \cup \{\nu_T^\circ\}_T \cup \{\epsilon_x\} \cup \{\rho_x\} \cup \{\sigma_x\} \cup \{\tau_x\}
\]
is locally confluent, terminating, and facially compatible (as $\mc D_x$-subdivisions).
\end{prop}

\begin{proof}
We first check local confluence. Note that for any simplices $T_1$, $T_2$, the sets $\mc M_{T_1}^\circ$, $\mc N_{T_2}^\circ$, $\mc E_x$, $\mc R_x$, $\mc S_x$, and $\mc T_x$ are pairwise disjoint. Thus, if $A \in \mc D_x$ and there are two distinct moves from $\{A\}$, then these moves must be either $\mu_{T_1}$, $\mu_{T_2}$ for some $T_1$, $T_2$ or $\nu_{T_1}$, $\nu_{T_2}$ for some $T_1$, $T_2$. The same argument from Proposition~\ref{prop:confgammaT} shows that the results of these moves are joinable.

We next show the terminating property. Let $A \in \mc D_x$ with standard form $(p,S,a) \times (q,S,b) \times k$, and suppose we have a subdivision $f(A)$ where $f \in \ms S_x$. Let $B \in \Max f(A)$ with standard form $(p',S',a') \times (q',S',b') k'$. Then one of the following holds:
\begin{enumerate}
\item $\sum_{j \le k'} \dim S'_j < \sum_{j \le k} \dim S_j$.
\item The above inequality is equality, and
\[
\sum_{\substack{i \\ j \le k'}} a'_{ij} < \sum_{\substack{i \\ j \le k}} a_{ij}.
\]
\item The above two inequalities are equality, and $\abs{p'} < \abs{p}$. (This can only possibly occur if $f = \epsilon$.) 
\end{enumerate}
It follows that $\ms S_x$ is terminating.

Finally, we note that Propositions~\ref{prop:facemuT}, \ref{prop:facenuT}, \ref{prop:faceepsilon}, \ref{prop:facetaux}, \ref{prop:facesigmax}, and \ref{prop:facerhox} imply the first criterion of facial compatibility. For the second criterion, let $A \in \mc D_x$ with standard form $(p,S,a) \times (q,S,b) \times k$. We have that $A$ is terminal if and only if $A$ is not in any of the sets $\mc M^\circ$, $\mc N^\circ$, $\mc E_x$, $\mc R_x$, $\mc S_x$, $\mc T_x$. This occurs if and only if $\abs{p} = 1$ and $k = 0$. Clearly if this property holds for $A$, then it holds for any $B \le_{\mc D} A$. Hence $\ms S_x$ is facially compatible.
\end{proof}

From this we can conclude the following.

\begin{thm} \label{thm:Deltax}
There is a canonical subdivision $\Delta_x : \mc D_x \to \Subd(\mc D_x)$ over $2\mc P$ such that for all $A \in \mc D_x$ and $B \in \Max(\Delta_x(A))$, where $(p,S,a) \times (q,S,b) \times k$ is the standard form of $B$, we have the following:
\begin{enumerate}[label=(\alph*)]
\item $\abs{p} = 1$ and $k = 0$.
\item If $A \in \mc D_x^\circ$, then $L(B) = L(A)$.
\item If $A \in \mc D_x^\bullet$, then $L(B) < L(A)$.
\end{enumerate}
\end{thm}

\begin{proof}
Construct $\Delta_x$ from $\ms S_x$ and Theorem~\ref{thm:confluence}. Then, (a) follows from the proof of Proposition~\ref{prop:S}. (b)  follows from Propositions~\ref{prop:LmuT}, \ref{prop:LnuT}, and \ref{prop:Lepsilon}. For (c), note that if $B$ satsifies (a), then $x$ is not a box point of $B$. Thus, (c) follows from Propositions~\ref{prop:LmuT}, \ref{prop:LnuT}, \ref{prop:Lepsilon}, \ref{prop:Ltaux}, \ref{prop:Lsigmax}, and \ref{prop:Lrhox}.
\end{proof}

We are now ready for the final proof. As in Theorem~\ref{thm:main2}, let $d$ be the dimension of the space we are working in and let $c$ be an integer $ \ge d!+d$.

\begin{proof}[Proof of Theorem~\ref{thm:main2}]
For any $\mc D$-complex $X$ with $\dim \abs{X} = d$, we consider the collection of lattices
\[
\ms L(X) := \{ L(A) : A \in \Max X \}.
\]
Start with a $\mc D$-subdivision $X$ with $\dim \abs{X} = d$ and 2-support $(P,Q)$. Assume that every element of $X$ is terminal with respect to $\ms S_x$. Consider a map $\theta : X \to \mc D$ defined as follows: If $A \in X$ has standard form $(p,S,a) \times (q,S,b) \times 0$, then
\[
\theta(A) := (cp, S, ca) \times (d!q, S, d!b) \times \abs{S}.
\]
Then $\theta(X)$ is a $\mc D$-subdivision with 2-support $(cP, d!Q)$.

Choose some lattice $L \in \ms L(X)$ and some nonzero $x \in \bb R^d / L$. Note that $\theta(X)$ is a $\mc D_x$-complex; indeed, for all $j \le \abs{S}$, we have $b_{1j} = 0$ or $d!$ and $a_{1j} = c \ge d! + d \ge b_{1j} + d$. Let $X' := \Delta_x^\ast(\theta(X))$, where $\Delta_x$ is from Theorem~\ref{thm:Deltax} and the $\ast$ construction is from Proposition~\ref{prop:canonicalrefinement}. Then $X'$ is a $\mc D$-subdivision with 2-support $(cP, d!Q)$ where every element is terminal with respect to $\ms S_x$. Comparing $\ms L(X)$ and $\ms L(X')$, by Theorem~\ref{thm:Deltax}, we have that $\ms L(X')$ is obtained by replacing at least one lattice of $\ms L(X)$ with lattices of lower index, while keeping the other lattices the same.
Thus, if we repeat the above process on $X'$ instead of $X$, and so on, we will eventually obtain a $\mc D$-subdivision $Y$ with 2-support $(c^N P, (d!)^N Q)$ for some $N$, such that every element of $Y$ is terminal and $\ms L(Y) = \{ \bb Z^d \}$.

Let $r$, $s$ be nonnegative integers. Consider the map $\omega_{r,s} : Y \to \mc C$ given by
\[
\omega((p,S,a) \times (q,S,b) \times 0) = (rp+sq,S,ra+sb).
\]
Then $\omega_{r,s}(Y)$ is a $\mc C$-subdivision with support $rc^N P + s(d!)^N Q$. Let
\[
Z := \Gamma^\ast(\omega_{r,s}(Y)),
\]
where $\Gamma$ is from Theorem~\ref{thm:Gamma}. Then $Z$ is a $\mc C$-subdivision with support $rc^N P + s(d!)^N Q$ such that for all $(p,S,a) \in Z$ in standard form, we have $\abs{S} = 0$ and $L(p,S,a) = \bb Z^d$. Thus $\Cay(Z)$ is a unimodular triangulation of $rc^N P + s(d!)^N Q$.

Now, let $P$ be a $d$-dimensional integral polytope in $\bb R^d$, and let $X_0$ be any triangulation of $P$ into integral simplices. Let
\[
X := \{ ((0),(T),(1)) \times ((0),(T),(1)) \times 0 \in \mc D : T \in X_0 \}
\]
Then $X$ is $\mc D$-subdivision with 2-support $(P,P)$, all of whose elements are terminal. Hence, applying the above argument to $X$ gives the result.
\end{proof}

\end{document}